\documentclass{amsart}[12pt]
\pdfoutput=1

\usepackage{tikz}
\usepackage[sc]{mathpazo}
\usepackage{hyperref}
\usepackage{microtype}
\usepackage{tikz}
\usetikzlibrary{snakes, arrows}
\usepackage{graphicx}
\usepackage{mathtools} 
\usepackage{xcolor}
\hypersetup{
    colorlinks,
    linkcolor={red!80!black},
    citecolor={blue!80!black},
    urlcolor={blue!80!black}
}

\theoremstyle{definition}

\newtheorem*{theorem*}{Theorem}
\newtheorem{dummy}{}[section]
\newtheorem{theorem}[dummy]{Theorem}
\newtheorem{lemma}[dummy]{Lemma}
\newtheorem{example}[dummy]{Example}
\newtheorem{corollary}[dummy]{Corollary}
\newtheorem{definition}[dummy]{Definition}
\newtheorem{remark}[dummy]{Remark}
\newtheorem{proposition}[dummy]{Proposition}

\newtheorem{conjecture}[dummy]{Conjecture}
\newtheorem*{FalseHope}{False Hope}

\newcommand{\Z}{\mathbb{Z}}
\newcommand{\C}{\mathbb{C}}
\newcommand{\N}{\mathbb{N}}
\newcommand{\R}{\mathbb{R}}
\newcommand{\proj}{\mathbb{P}}
\newcommand{\core}{\mathbf{core}}

\newcommand{\SC}{\mathbf{SC}}
\newcommand{\TD}{\mathbf{TD}}
\newcommand{\OC}{\mathcal{OC}}
\newcommand{\Cat}{\mathbf{Cat}}
\newcommand{\sk}{s\ell}
\newcommand{\vac}{\mathbf{vac}}
\newcommand{\inv}{\mathbf{inv}}

\newcommand{\INV}{\mathbf{INV}}
\newcommand{\DES}{\mathbf{DES}}
\newcommand{\des}{\mathbf{des}}
\newcommand{\maj}{\mathbf{maj}}
\newcommand{\siz}{\mathbf{siz}}
\newcommand{\sqin}{\mathbf{sqin}}
\newcommand{\LD}{\mathbf{LD}}
\newcommand{\VS}{\mathbf{VS}}
\newcommand{\cone}{\mathfrak{c}}
\newcommand{\polyh}{\mathfrak{c}}
\newcommand{\dominant}{\mathcal{D}}
\DeclareMathOperator{\Sym}{Sym}

\begin{document}
\title{Lattice points and simultaneous core partitions}
\begin{abstract}
We observe that for $a$ and $b$ relatively prime, the ``abacus construction'' identifies the set of simultaneous $(a,b)$-core partitions with lattice points in a rational simplex. Furthermore, many statistics on $(a,b)$-cores are piecewise polynomial functions on this simplex.
We apply these results to rational Catalan combinatorics. Using Ehrhart theory, we reprove Anderson's theorem \cite{anderson} that there are $(a+b-1)!/a!b!$ simultaneous $(a,b)$-cores, and using Euler-Maclaurin theory we prove Armstrong's conjecture \cite{AHJ} that the average size of an $(a,b)$-core is $(a+b+1)(a-1)(b-1)/24$. Our methods also give new derivations of analogous formulas for the number and average size of self-conjugate $(a,b)$-cores.
We conjecture a unimodality result for $q$ rational Catalan numbers, and make preliminary investigations in applying these methods to the $(q,t)$-symmetry and specialization conjectures. We prove these conjectures for low degree terms and when $a=3$, connecting them to the Catalan hyperplane arrangement and quadratic permutation statistics.
\end{abstract}
\author{Paul Johnson}
\address{University of Sheffield}
\email{paul.johnson@sheffield.ac.uk}

\maketitle
\section{Introduction}
This paper establishes lattice point geometry as a foundation for the study of simultaneous core partitions, and, more generally, rational Catalan numbers, which, by a theorem of Anderson, count simultaneous cores.

Rational Catalan numbers and their $q$ and $(q,t)$ analogs are a natural generalization of Catalan numbers.  Apart from their intrinsic combinatorial interest, they appear in the study of Hecke algebras \cite{GG}, affine Springer varieties \cite{LS}, and compactified Jacobians of singular curves \cite{EM1, EM2}.   

Lattice point geometry provides a unified approach to proving many known results about simultaneous core partitions, such as Anderson's result, and also lets us prove a conjecture of Armstrong about the average size of simultaneous core partitions.  Furthermore, lattice point techniques provide an avenue to attack the specialization and symmetry conjectures for $(q,t)$-rational Catalan numbers, the most important open questions in the field.  

Our first result is to reprove Anderson's theorem by identifying simultaneous core partitions with lattice points in a rational simplex.  The connection between rational Catalan numbers and this simplex is not new; it appears for instance in \cite{LS, GMV}.    However, we are not aware of any work using this connection to apply lattice point techniques.  After this identification is made, many results follow quite naturally.
  
\subsection{Background: Simultaneous cores and rational Catalan numbers}
A \emph{partition of $n$} is a nonincreasing sequence $\lambda_1\geq \lambda_2\geq
\lambda_{k}> 0$ of positive integers such that $\sum \lambda_i=n$. We
call $n$ the $\emph{size}$ of $\lambda$ and denote it by $|\lambda|$; we call
$k$ the \emph{length} of $\lambda$ and denote it by $\ell(\lambda)$.
\subsubsection{Hooks and Cores}
We frequently identify $\lambda$ with its Young diagram, in
English notation -- that is, we draw the parts of $\lambda$ as
the columns of a collection of boxes.
\begin{definition}
The \emph{arm} $a(\square)$ of a cell $\square$ is the number of cells contained in $\lambda$ and above $\square$, and the \emph{leg} $l(\square)$ of a cell is the number of cells contained in $\lambda$ and to the right of $\square$.
The \emph{hook length} $h(\square)$ of a cell is $a(\square)+l(\square)+1$.
\end{definition}
\begin{example}
The cell $(2,1)$ of $\lambda=3+2+2+1$ is marked $s$; the cells in the
leg and arm of $s$ are labeled $a$ and $l$, respectively.
\begin{center}
\begin{tikzpicture}[scale=.5]
\draw[thin, gray] (0,0) grid (1,3);
\draw[thin, gray] (1,0) grid (3,2);
\draw[thick] (0,0)--(0,3)--(1,3)--(1,2)--(3,2)--(3,1)--(4,1)--(4,0)--cycle;
\draw (1.5,.5) node{$s$};
\draw (1.5,1.5) node{$a$};
\draw (2.5,.5) node{$l$};
\draw (3.5,.5) node {$l$};
\draw (8.5,1.5) node[align=left] {$a(s)=\# a=1$ \\ $l(s)=\# l=2$ \\ $h(s)=4$};
\end{tikzpicture}
\end{center}
\end{example}
We now introduce our main object of study.
\begin{definition}
An $a$-\emph{core} is a partition that has no hook lengths of size $a$. An $(a,b)$-\emph{core} is a partition that is simultaneously an $a$-core and a $b$-core.  We use $\mathcal{C}_{a,b}$ to denote the set of $(a,b)$-cores.
\end{definition}
\begin{example}
We have labeled each cell $\square$ of $\lambda=3+2+2+1$ with its hook length $h(\square)$.
\begin{center}
\begin{tikzpicture}[scale=.5]
\draw[thin, gray] (0,0) grid (1,3);
\draw[thin, gray] (1,0) grid (3,2);
\draw[thick] (0,0)--(0,3)--(1,3)--(1,2)--(3,2)--(3,1)--(4,1)--(4,0)--cycle;
\draw (.5,.5) node{$6$};
\draw (1.5,.5) node{$4$};
\draw (2.5,.5) node{$3$};
\draw (3.5,.5) node{$1$};
\draw (.5,1.5) node {$4$};
\draw (1.5,1.5) node{$2$};
\draw (2.5,1.5) node{$1$};
\draw (.5,2.5) node{$1$};
\end{tikzpicture}
\end{center}
We see that $\lambda$ is \emph{not} an $a$-core for $a\in \{1,2,3,4,6\}$;
but it \emph{is} an $a$-core for all other $a$.
\end{example}
\subsubsection{Rational Catalan numbers}
Recall that the Catalan number $\Cat_n=\frac{1}{2n+1}\binom{2n+1}{n}$. Catalan numbers count hundreds of different combinatorial objects; for example, the number of lattice paths from $(0,n)$ to $(n+1,0)$ that stay strictly below the line connecting these two points.
Rational Catalan numbers are a natural two parameter generalization of $\Cat_n$.
\begin{definition}
For $a,b$ relatively prime, the \emph{rational Catalan number}, or \emph{$(a,b)$ Catalan number} $\Cat_{a,b}$ is
$$\Cat_{a,b}=\frac{1}{a+b}\binom{a+b}{a}$$
\end{definition}
The rational Catalan number $\Cat_{a,b}$ counts the number of lattice paths from $(0,a)$ to $(b,0)$ that stay beneath the line from $(0,a)$ to $(b,0)$. This is consistent with the specialization $\Cat_{n,n+1}=\Cat_n$.

\subsection{}
Simultaneous cores and rational Catalan numbers are connected by:
\begin{theorem}[Anderson \cite{anderson}] \label{thm:anderson}
If $a$ and $b$ are relatively prime, then $|\mathcal{C}_{a,b}|=\Cat_{a,b}$; that is, rational Catalan numbers count $(a,b)$-cores.
\end{theorem}

Our first result is a new proof of Theorem \ref{thm:anderson} using the geometry of lattice points in rational polyhedra.  This framework easily extends to prove other results c,hief among them a proof of Armstrong's conjecture:

\begin{theorem} \label{thm:armstrong}
The average size of an $(a,b)$-core is $(a+b+1)(a-1)(b-1)/24$.
\end{theorem}

\begin{remark}
 Armstrong conjectured Theorem \ref{thm:armstrong} in 2011; it first appeared in print in \cite{AHJ}.
Stanley and Zanello \cite{SZ} have proven the Catalan case ($a=b+1$) of Armstrong's conjecture by different methods, and building on their work Aggarwal \cite{Aggarwal} has proven the case $a=mb+1$.
\end{remark}

Our two main tools are the abacus construction and Ehrhart theory. We briefly recall these ideas before giving a high-level overview of the proofs of Theorems \ref{thm:anderson} and \ref{thm:armstrong}.

\subsubsection{Abaci}
The main tool used to study $a$-cores is the ``abacus construction''. We review this construction in detail in Section \ref{sec:abacus}. For now, we note that there are at least two variants of the abacus construction in the literature.
The first construction, which we call the \emph{positive abacus}, gives a bijection between $a$ core partitions and $\N^{a-1}$. Anderson's original proof used the positive abacus as part of a bijection between $(a,b)$-cores and $(a,b)$-Dyck paths, which were already known to be counted by $\Cat_{a,b}$.
We use the second construction, which we call the \emph{signed abacus}. The signed abacus is a bijection between $a$-core partitions and points in the $a-1$ dimensional lattice
$$\Lambda_a=\left\{c_1,\dots,c_a\in\Z\Big |\sum c_i=0\right\}$$

Key to our proof of Armstrong's conjecture is

\begin{theorem*}
Under the signed abacus bijection, the size of an $a$-core is given by the quadratic function
$$Q(c_1,\dots, c_a)=\frac{a}{2}\sum c_i^2+\sum ic_i$$
\end{theorem*}

We are not sure where exactly where this theorem originates; a stronger version is used in \cite{GKS} and \cite{DS} to prove generating functions counting certain partitions are modular forms.   For completeness, we give a proof as Theorem \ref{thm:quadratic}.

\subsubsection{Ehrhart / Euler-Maclaurin}
The number of lattice points in a polytope can be viewed as a discrete version of the volume of a polytope. Ehrhart theory is the study of this analogy. A gentle introduction to Ehrhart theory may be found in \cite{BR}.
Let $V$ be an $n$ dimensional real vector space, and $\Lambda\subset V$ an $n$ dimensional lattice. Conretely, $\Lambda=\Z^n, V=\R^n$. A lattice polytope $P\subset V$ is a polytope all of whose vertices are points of $\Lambda$.
For $t$ a positive integer, let $tP$ denote the polytope obtained by scaling $P$ by $t$. For $t\geq 0$, define $L(P,t)$ to the number of lattice points in $tP$:
$$L(P,t)=\#\{\Lambda\cap tP\}$$
Clearly, the volume of $tP$ is $t^n$ times the volume of $P$. Ehrhart showed that, parallel to this fact, $L(P,t)$ is a degree $n$ polynomial in $t$.

Other than the fact that $L(P,t)$ is a polynomial of degree $n$, the one fact from Ehrhart theory we use is Ehrhart reciprocity. If we scale a polytope by $-t$, then keeping track of orientation the volume changes by $(-t)^n$. The polynomial $L(P,t)$ is not in general even or odd, and so $L(P,-t)$ cannot be $(-1)^n$ times the number of lattice points in $-P$. Ehrhart reciprocity states that instead 
$$L(P,-t)=(-1)^nL(P^\circ, t)$$
 where $P^\circ$ denotes the interior of $P$.

The results of Ehrhart theory extend to an analogy between integrating a polynomial over a region and summing it over the lattice points in a polytope. This is an extension of the familiar ``sum of the first $n$ cubes'' formulas.  Specifically, if $f$ is a polynomial of degree $d$ on $V$, then $\int_{tP} f$ is a polynomial of degree $d+n$. Euler-Maclaurin theory says that the discrete analog
$$L(f, P, t)=\sum_{x\in\Lambda\cap tP} f(x)$$
is also a polynomial of degree $d+ n$. Ehrhart reciprocity also extends : $$L(f,P,-t)=(-1)^nL(f, -P^\circ, t)$$
Although unsurprising to experts, apparently this extension was first used (without proof) in \cite{CJM}; a proof now appears in \cite{AB}.

\subsubsection{Initial motivation}
To explain the method used to prove Theorems \ref{thm:anderson} and \ref{thm:armstrong}, we begin with the following

\begin{FalseHope}
Fix $a$. Under the signed abacus construction, the set of $(a,b)$-cores are exactly those lattice points in $bP$, for some integral polytope $P\subset V_a$.
\end{FalseHope}

If the false hope were true, Ehrhart theory would imply that, for $b$ relatively prime to a fixed $a$, $|\mathcal{C}_{a,b}|$ would be a polynomial of degree $a-1$ in $b$. It is clear from the definition that this polynomiality property holds for $\Cat_{a,b}$. Thus, proving Anderson's theorem for a fixed $a$ would reduce to showing that two polynomials are equal, which only requires checking finitely many values.

Furthermore, it is known that the size of an $a$-core is a quadratic function $Q$ on the lattice. Thus, if the False Hope were true Euler-Maclaurin theory would give that the total size of all $(a,b)$ cores was a polynomial of degree $a+1$ in $b$, and again we could hope to exploit this polynomiality in a proof.

\subsubsection{}
The False Hope is not quite true, but the strategy outlined above is essentially the one we follow. The set of $b$ cores inside the lattice of $a$ cores is a simplex, which we call $\SC_a(b)$ for \emph{Simplex of Cores}.
One minor tweak needed to the False Hope is that as we vary $b$ $\SC_a(b)$ is not only scaled, but also changed by a linear transformation. These transformations preserve the number of lattice points and the quadratic function $Q$ giving the size of the partitions, and so do not pose any real difficulties.
More troubling is that the polytope $\SC_{a}(b)$ is not integral, but only rational. Recall that a polytope $P$ is \emph{rational} if there is some $k\in\Z$ so that $kP$ is a lattice polytope.

\subsubsection{Rational Polytopes and quasipolynomials}
Ehrhart and Euler/Maclaurin theory can be extended to rational polytopes at the cost of replacing polynomials by \emph{quasipolynomial}.

\begin{definition}
A function $f:\Z\to\C$ is a quasipolynomial of degree $d$ and period $n$ if there exist $n$ polynomials $p_0,\dots, p_{n-1}$ of degree $d$, so that for $x\in k+n\Z$, we have $f(x)=p_k(x)$.
\end{definition}

\begin{example} Let $P$ be the polytope $x, y\geq 0, 2x+y\leq 1$. Then
$$\#\{tP\cap \Z^2\}=\left\{\begin{array}{rl} \frac{t^2+4t+4}{4} & \text{$t$ even} \\
\frac{t^2+4t+3}{4} & \text{$t$ odd}\end{array}\right.$$
\end{example}

Since $\Cat_{a,b}$ is defined only for $a$ and $b$ relatively prime, it fits nicely into the quasipolynomial framework. For $a$ fixed, and $b$ in a fixed residue class mod $a$, $\Cat_{a,b}$ is a polynomial. It just so happens that residue classes relatively prime to $a$ have identical polynomials. Such ``accidental'' equalities between the polynomials for different residue classes happen frequently in Ehrhart theory, but are mysterious in general. Perhaps the most studied manifestation of this is \emph{period collapse} (see \cite{Haase} and references), where the quasipolynomial is in fact a polynomial. In our case, symmetry considerations give an elementary explanation of the ``accidental'' equalities between the polynomials for different residue classes.

\subsubsection{}
In Lemma \ref{lem:standardsimplex} we show that the the polyhedron $\SC_a(b)$ is isomorphic to a rational simplex we call $\TD_a(b)$ (for \emph{Trivial Determinant}) that we now describe. Let $L_k$ be the one dimensional representation of $\Z_a$ where $1\in \Z_a$ acts as $\exp(2\pi ik/a)$. Then any $b$ dimensional representation $V$ of $\Z_a$ may be written as
$$V=\bigoplus_{k=0}^{a-1} L_k^{\oplus z_k}$$
for nonnegative integers $z_i$ satisfying $\sum z_i=b$. Thus, there is a bijection between the set of $b$ dimensional representations of $\Z_a$ and the standard simplex $b\Delta_{a-1}$, which has $\binom{a-1+b}{b}$ lattice points.

The simplex $\TD_a(b)$ is obtained by considering only those representations that have trivial determinant (i.e., $\wedge^b V\cong L_0$), or equivalently restricting to the index $a$ sublattice given by $\sum iz_i=0\pmod a$.

More generally, consider the set of representations with determinant isomorphic to $L_k$ for any $k$. Tensoring $V$ by $L_1$ corresponds to the cyclic permutation of coordinates $z_k\mapsto z_{k+1}$, and changes the determinant of $V$ by tensoring by $L_b$ (where we are using periodic indices). Thus, the dual $\Z_a$ acts on the set of all $b$ dimensional representations of $\Z_a$, and when $b$ is relatively prime to $a$ this action is free, and each orbit contains exactly one representation with trivial determinant. Hence, the number of points in $\TD_a(b)$ is exactly one $a$th of the number of points in $b\Delta_a$, namely $\binom{a-1+b}{b}/a=\Cat_{a,b}$.  

Thus the identification of $\SC_a(b)$ and $\TD_a(b)$ reproves Anderson's theorem. With some more work, Armstrong's conjecture follows in a similar manner.

The situation is illustrated in Figure \ref{fig:sublattice}. The left hand picture shows $\TD_3(10)\cong \SC_3(10)$, while the right hand picture shows the standard simplex $10\Delta_2$. The black dots are the representations with trivial determinant, while the red and green dots are those representations with determinant $L_1$ and $L_2$. Rotating about the blue circle by 120 degrees corresponds to tensoring by $L_1$ and permutes the different colored dots.
\begin{figure}[h!]
\caption{$\Lambda_3$ and $\Lambda_3^\prime$ inside $\mathcal{C}_{3, 10}$}
\label{fig:sublattice}
\begin{tikzpicture}[scale=.8]
\begin{scope}
\clip (-3.5, -4*1.7320508/2)--(-3.5, 4*1.7320508/2)--(2.5,
4*1.7320508/2)--(2.5, -4*1.7320508/2)--cycle;
\fill[gray!50] (-3,0)--(2, 5/3*1.7320508)--(2, -5/3*1.7320508)--cycle;
\foreach \x in {-7,-6,...,7}{
\foreach \y in {-7,-6,...,7}{
\node[draw,circle,inner sep=1pt,fill] at (\x-\y/2,1.7320508*\y/2) {};}}
\draw[dashed] (2, 10*1.7320508/2)--(2, -5*1.7320508);
\draw[dashed] (-15-3, 5*1.7320508)--(15-3, -5*1.7320508);
\draw[dashed] (-15-3, -5*1.7320508)--(15-3, 5*1.7320508);
\end{scope}
\begin{scope}[xshift=9cm]
\clip (-3.5, -4*1.7320508/2)--(-3.5, 4*1.7320508/2)--(2.5,
4*1.7320508/2)--(2.5, -4*1.7320508/2)--cycle;
\fill[gray!50] (-3,0)--(2, 5/3*1.7320508)--(2, -5/3*1.7320508)--cycle;
\draw[thick, blue] (.333,0) circle(.08);
\foreach \x in {-7,-6,...,7}{
\foreach \y in {-7,-6,...,7}{
\node[draw,circle,inner sep=1pt,fill] at (\x-\y/2,1.7320508*\y/2) {};}}
\begin{scope}[green, xshift=.5cm, yshift=.2886751cm]
\foreach \x in {-7,-6,...,7}{
\foreach \y in {-7,-6,...,7}{
\node[draw,circle,inner sep=.7pt,fill] at (\x-\y/2,1.7320508*\y/2) {};}}
\end{scope}
\begin{scope}[red, xshift=.5cm, yshift=-.2886751cm]
\foreach \x in {-7,-6,...,7}{
\foreach \y in {-7,-6,...,7}{
\node[draw,circle,inner sep=.7pt,fill] at (\x-\y/2,1.7320508*\y/2) {};}}
\end{scope}
\draw[dashed] (2, 10*1.7320508/2)--(2, -5*1.7320508);
\draw[dashed] (-15-3, 5*1.7320508)--(15-3, -5*1.7320508);
\draw[dashed] (-15-3, -5*1.7320508)--(15-3, 5*1.7320508);
\end{scope}
\end{tikzpicture}
\end{figure}

\subsubsection{Self-conjugate simultaneous cores}
The lattice point technique easily adapts to treat the case of self-conjugate simultaneous cores. Ford, Mai and Sze have shown \cite{FMS} that self-conjugate $(a,b)$-core partitions are counted by
$$\binom{\big\lfloor \frac{a}{2} \big\rfloor +\big\lfloor \frac{b}{2} \big\rfloor}{\big\lfloor \frac{a}{2} \big\rfloor}$$

Armstrong conjectured, and Chen, Huang and Wang recently proved \cite{CHW}, that the average size of self-conjugate $(a,b)$-core partitions is the same as the average size of all $(a,b)$-core partitions, namely $(a+b+1)(a-1)(b-1)/24$.

In Section \ref{sec:conjugate} we give new proofs of both of these results. A key idea is that the action of conjugation on $\SC_a(b)$ corresponds to the action of taking dual representations on $\TD_a(b)$.

\subsection{The \texorpdfstring{$q$}{q}-analog}
Section \ref{sec:qcat} applies the lattice point framework to $q$-rational Catalan numbers.

\subsubsection{$q$-analogs}
The $q$-rational Catalan numbers $\Cat_{a,b}(q)$ are defined by the obvious $q$-analog of $\Cat_{a,b}$:
$$\Cat_{a,b}(q)=\frac{1}{[a+b]_q}{a+b \brack a}_q$$
It is nontrivial that the coefficients of $\Cat_{a,b}(q)$ are positive \cite{GG}. The main question we pursue in Section \ref{sec:qcat} is whether our lattice point view can shed any light on this positivity question.

An obvious hope is that $\Cat_{a,b}(q)$ is a sum over the lattice points in $\SC_{a,b}$ of $q^L$, where $L$ is some linear function. This does not appear to be true. However, we conjecture that there is an index $a^{a-2}$ sublattice $\Lambda^\prime$ of the lattice of cores, and a function $\iota$ on the cosets $\mathfrak{c}$ of $\Lambda^\prime$, so that $\Cat_{a,b}(q)$ is the sum over the lattice points in $\SC_{a,b}$ of $q^{L+\iota}$; this would give an expression for $\Cat_{a,b}(q)$ as a sum of $a^{a-2}$ terms of the form
$$q^{\iota(\mathfrak{c})}{S(\mathfrak{c}) \brack a-1}_{q^a}.$$
which would explain the positivity of the coefficients of $\Cat_q(a,b)$.

Furthermore, this conjectural formula leads naturally to a unimodality conjecture about $\Cat_{a,b}(q)$. Recall that a sequence $a_1,\dots, a_n$ is \emph{unimodal} if there is some $k$ so that $$a_1\leq a_2\leq \dots \leq a_{k-1}\leq a_k \geq a_{k-1} \geq a_{k-2}\geq\dots\geq a_n$$
The coefficients of $\Cat_{a,b}$ are not unimodal. However, we conjecture that, if we fix $0
\leq k <a$, and look only at the coefficients of $\Cat_{a,b}(q)$ of the form $q^{an+k}$, the resulting sequences are unimodal.

\subsection{The \texorpdfstring{$(q,t)$}{(q,t)}-analog}
Armstrong, Hanusa and Jones \cite{AHJ} have defined a $(q,t)$-analog of $\Cat_{a,b}$ by counting $(a,b)$-cores according to length $\ell$ and co-skewlength $\sk^\prime$, and have made a \emph{symmetry} and a \emph{specialization} conjecture about $\Cat_{a,b}(q,t)$. Section \ref{sec:qt} uses lattice point techniques to make progress toward these conjectures.

\subsubsection{Definition and conjectures}
We first introduce the skew length statistic needed to define $(q,t)$-rational Catalan numbers.

\begin{definition}
Let $a<b$ be relatively prime, and $\lambda$ an $(a,b)$-core. The \emph{$b$-boundary} of $\lambda$ consists of all cells $\square\in\lambda$ with $h(\square)<b$.

Group the parts of $\lambda$ into $a$ classes, according to $\lambda_i-i\mod a$; (note, at least one of the $a$ classes is empty since $\lambda$ is an $a$-core). The \emph{$a$-parts of $\lambda$} consist of the maximal $\lambda_i$ among each of the $i$ residue classes.

The \emph{skew length of $\lambda, \sk(\lambda)$} is the number of cells of $\lambda$ that are in an $a$-row of $\lambda$ and in the $b$-boundary of $\lambda$. The \emph{co-skew-length} $\sk^\prime(\lambda)$ is $(a-1)(b-1)/2-\sk(\lambda)$. Note that the skew length depends on $a$ and $b$; where necessary, where we refer to the $(a,b)$ skew length.
\end{definition}

\begin{definition}
Let $a<b$ be relatively prime. The \emph{$(q,t)$-rational Catalan number} is
$$\Cat_{a,b}(q,t)=\sum_{\lambda\in\mathcal{C}_{a,b}} q^{\ell(\lambda)}t^{\sk^\prime(\lambda)}$$
\end{definition}

\begin{example}
We illustrate that the $(3,11)$ skew length of the partition $\lambda=9+7+5+3+2+2+1+1$ is $9$. Each cell of $\lambda$ is labeled with its hook length; if the cell is part of the $11$-boundary we have written the text in blue. Beneath each part of $\lambda$ we have written $\lambda_i-i\pmod 3$; if the parts is a $3$-part of $\lambda$ we have written it in red. The cells that contribute to the skew length have been shaded light green.
\begin{center}
\begin{tikzpicture}[scale=.5]
\fill[green!20!white] (0,2) rectangle (1,9);
\fill[green!20!white] (4,0) rectangle (5,2);
\draw[thin, gray] (0,0) grid (1,8);
\draw[thin, gray] (0,0) grid (2,6);
\draw[thin, gray] (0,0) grid (3,4);
\draw[thin, gray] (0,0) grid (6,2);
\draw[thin, gray] (0,0) grid (7,1);
\draw[thick] (0,0)--(0,9)--(1,9)--(1,7)--(2,7)--(2,5)--(3,5)--(3,3)--(4,3)--(4,2)--(6,2)--(6,1)--(8,1)--(8,0)--cycle;
\draw (.5, .5) node{16};
\draw ( .5,1.5) node{13 };
\draw[blue] ( .5,2.5) node{10};
\draw[blue] ( .5,3.5) node{8 };
\draw[blue] ( .5,4.5) node{7 };
\draw[blue] ( .5,5.5) node{5 };
\draw[blue] ( .5,6.5) node{4 };
\draw[blue] ( .5,7.5) node{2 };
\draw[blue] ( .5,8.5) node{1 };
\draw (1.5, .5) node{13};
\draw[blue] (1.5, 1.5) node{10};
\draw[blue] (1.5, 2.5) node{ 7};
\draw[blue] (1.5, 3.5) node{ 5};
\draw[blue] (1.5, 4.5) node{ 4};
\draw[blue] (1.5, 5.5) node{ 2};
\draw[blue] (1.5, 6.5) node{ 1};
\draw[blue] (2.5, .5) node{10};
\draw[blue] (2.5, 1.5) node{ 7};
\draw[blue] (2.5, 2.5) node{ 4};
\draw[blue] (2.5, 3.5) node{ 2};
\draw[blue] (2.5, 4.5) node{ 1};
\draw[blue] (3.5, .5) node{7 };
\draw[blue] (3.5, 1.5) node{4 };
\draw[blue] (3.5, 2.5) node{1 };
\draw[blue] (4.5, .5) node{5 };
\draw[blue] (4.5, 1.5) node{2 };
\draw[blue] (5.5, .5) node{4 };
\draw[blue] (5.5, 1.5) node{1 };
\draw[blue] (6.5, .5) node{2 };
\draw[blue] (7.5, .5) node{1 };
\draw[red] (.5, -.5) node{2};
\draw (1.5, -.5) node{2};
\draw (2.5, -.5) node{2};
\draw (3.5, -.5) node{2};
\draw[red] (4.5, -.5) node{0};
\draw (5.5, -.5) node{2};
\draw (6.5, -.5) node{0};
\draw (7.5, -.5) node{2};
\end{tikzpicture}
\end{center}
\end{example}

There are two main conjectures about $\Cat_{a,b}(q,t)$:

\begin{conjecture}[Specialization]
$$\sum_{\lambda\in \mathcal{C}_{a,b}} q^{\ell(\lambda)+\sk(\lambda)}=q^{(a-1)(b-1)/2}\Cat_{a,b}(q,1/q)=\Cat_{a,b}(q)$$
\end{conjecture}

\begin{conjecture}[Symmetry]
$$\Cat_{a,b}(q,t)=\Cat_{a,b}(t,q)$$
\end{conjecture}

Our first result in Section \ref{sec:qt} is that $\ell$ and $\sk$ are piecewise linear functions on the simplex of cores $\SC_a(b)$, with walls of linearity given by the Catalan arrangement. This linearity means we can apply lattice point techniques to $\Cat_{a,b}(q,t)$; in particular, thereoms of Brion, Lawrence, and Varchenko.

\subsubsection{Brion, Lawrence, Varchenko}
Let $\mathcal{P}$ be a $d$-dimensional rational polytope. The \emph{enumerator function} $\sigma_{\mathcal{P}}(x)$ is a Laurent series whose monomials record the lattice points inside $\mathcal{P}$. Specifically:
$$\sigma_{\mathcal{P}}(x)=\sum_{\bf{m}\in \mathcal{P}\cap \Z^d}x^{\bf{m}}$$
where $x^{\bf{m}}=x_1^{m_1}x_2^{m_2}\cdots x_d^{m_d}$.
Theorems of Brion \cite{Brion}, Lawrence \cite{Lawrence}, and Varchenko \cite{Varchenko} (see \cite{BHS}) express $\sigma_{\mathcal{P}}(x)$ as a sum of rational functions determined by the cones at the vertices of $\mathcal{P}$.
\subsubsection{Rationality}
Since any count of the lattice points in $\mathcal{P}$ with respsect to linear functions is a specialization of the indicator function $\sigma_{\mathcal{P}}$, we may apply their theorems to each chamber of linearity of $\ell$ and $\sk$ to obtain expressions for $\Cat_{a,b}(q,t)$ as rational functions. 
\begin{proposition}
Fix $a$. Then $\Cat_{a,b}(q,t)$ has a uniform expression as a rational function in $q$ and $t$, with the dependence on $b$ only appearing in the exponents of the numerator. For $b$ in a fixed residue class modulo $a$, this dependence is linear.
\end{proposition}
\begin{example}{$\Cat_{3,b}(q,t)$}
In Proposition \ref{prop:qt3}, we explicitly compute this rational function when $a=3$. Write $b=3k+1+\delta$, where $\delta\in\{0,1\}$. Then
\begin{multline*}
\Cat_{a,b}(q,t)=\frac{t^{3k+\delta}}{(1-qt^{-1})(1-q^2t^{-1})} +\frac{q^kt^{k+1}+q^{k+\delta}t^{k+\delta}+q^{k+1}t^k}{(1-q^{-1}t^2)(1-q^2t^{-1})} \\+\frac{q^{3k+\delta}}{(1-tq^{-1})(1-t^2q^{-1})}
\end{multline*}
From this expression it is trivial to check the Symmetry and Specialization conjectures for $a=3$.
\end{example}
It is possible that this method could lead to a complete proof of the Symmetry and Specialization conjectures. This would require a thorough understanding of the geometry of the Catalan arrangement with respect to the shifted lattice $\Lambda+s$.
As a first step in this direction, we verify both conjectures for low degree terms for general $a$ and $b$. To do so we use a quadratic permutation statistic $\siz(\sigma)$. The Symmetry and Specialization conjectures in low degree reduce to the following formula for the joint distribution of the $\siz$ and $\maj$ statistics:
\begin{lemma}
$$\sum_{\sigma\in S_n} q^{\siz(\sigma)}t^{\maj(\sigma)}=\prod_{k=1}^n [k]_{q^{n+1-k}t}$$
\end{lemma}
After posting the initial preprint of this paper, I learned that quadratic permutation statistics have previously been considered by Bright and Savage \cite{BS} in the context of lecture hall permutations.  In particular, they had already proven the needed needed distribution if $\siz$ and $\maj$.

\subsection{Acknowledgements}
I learned about Armstrong's conjecture over dinner after speaking in the MIT combinatorics seminar. I would like to thank Jon Novak for the invitation, Fabrizio Zanello for telling me about the conjecture, and funding bodies everywhere for supporting seminar dinners.  

Thanks to Carla Savage for pointing out the previous occurence of quadratic permutation statistics.


\section{Abaci and Electrons} \label{sec:abacus}

This section is a review of the fermionic viewpoint of partitions and the abaci model of $a$-cores.  It contains no new material.  The main results are that $a$-cores are in bijection with points on the ``charge lattice'' $\Lambda_a$, and the size of a given $a$-core is given by a quadratic function on the lattice
.


\subsection{Fermions}
 We begin with a motivating fairy tale. It should not be mistaken for an attempt at accurate physics or accurate history.

\subsubsection{A fairy tale}
According to quantum mechanics, the possible energies levels of an electron are quantized -- they can only be half integers i.e., elements of $\Z_{1/2}=\{a+1/2|a\in\Z\}$.  In particular, basic quantum mechanics predicts electrons with negative energy.  Physically, it makes no sense to have negative energy electrons.

Dirac's \emph{electron sea} solves the problem of negative energy electrons by redefining the vacuum state $\vac$.  The Pauli exclusion principle states that each possible energy level can have at most one electron in it; thus, we can view any set of electrons as a subset $S\subset \Z_{1/2}$.  Intuitively, the vacuum state $\vac$ should consist of empty space with no electrons at all, and hence correspond to the set $S=\emptyset\subset\Z_{1/2}$.

Dirac suggested instead to take $\vac$ to be an
infinite ``sea'' of negative energy electrons.   Specifically, in Dirac's vacuum state every negative energy level is filled with an electron, but none of the positive energy states are filled. Then by Pauli's
exclusion principle we cannot add a negative energy electron
to $\vac$, but positive energy electrons can be added as usual.  Thus, Dirac's electron sea solves the problem of negative energy electrons.

As an added benefit, Dirac's electron sea predicts the positron, a particle
that has the same energy levels as an electron, but positive charge.  Namely, a positron corresponds to a ``hole'' in the electron sea, that is, a negative
energy level \emph{not} filled with an electron.  Removing a negative energy electron results in adding positive charge and positive energy, and hence can be interpreted as a having a positron.

\subsubsection{} Our fairytale leads us to the following definitions:

\begin{definition} Let $\Z_{1/2}^\pm$ denote the set of all positive/negative half integers, respectively.

The vacuum $\vac\subset \Z_{1/2}$ is the set $\Z_{1/2}^-$.

A \emph{state} $S$ is a set $S\subset\Z+1/2$ so that the symmetric
difference 
$$S \triangle \vac=(S\cap\Z_{1/2}^+)\cup (S^c\cap\Z_{1/2}^-)$$
 is finite.  States should be interpreted as a finite collection of electrons -- 
the elements of $S\cap \Z^+_{1/2}$, which we will denote by $S^+$, and a finite collections of positrons -- the elements of $S^c\cap \Z^-_{1/2}$, which we will denote by $S^-$.

The \emph{charge} $c(S)$ of a state $S$ is the number of positrons minus the number of electrons:
$$c(S)=| S^-| -| S^+|$$

The $\emph{energy}$ $e(S)$ of a state $S$ is the sum of all the energies of the positrons and the electrons:
$$e(S)=\sum_{k\in S^+}\; k +\sum_{k\in S^-}\; -k$$

\end{definition}

\subsubsection{Maya Diagrams}

It is convenient to have a graphical representation of states $S$.

The \emph{Maya diagram} of $S$ is an infinite sequence of circles on the $x$-axis, one circle centerred at each element of $\Z_{1/2}$, with the positive circles extending to the left and the negative direction to the right.  A black ``stone'' is placed on the
circle corresponding to $k\in\Z_{1/2}$ if and only if $k\in S$.

\begin{example}
The Maya diagram corresponding to the vacuum vector $\vac$ is shown below.
\begin{center}
\begin{tikzpicture}[scale=.6]
\draw (-5.5,0) node{$\cdots$};
\draw (-4.5,0) circle (.3) node[below=3pt]{$\frac{9}{2}$};
\draw (-3.5,0) circle (.3) node[below=3pt]{$\frac{7}{2}$};
\draw (-2.5,0) circle (.3) node[below=3pt]{$\frac{5}{2}$};
\draw (-1.5,0) circle (.3) node[below=3pt]{$\frac{3}{2}$};
\draw (-.5,0) circle (.3) node[below=3pt]{$\frac{1}{2}$};
\filldraw (.5,0) circle (.3) node[below=3pt]{$\frac{-1}{2}$};
\filldraw (1.5,0) circle (.3) node[below=3pt]{$\frac{-3}{2}$};
\filldraw (2.5,0) circle (.3) node[below=3pt]{$\frac{-5}{2}$};
\filldraw (3.5,0) circle (.3) node[below=3pt]{$\frac{-7}{2}$};
\filldraw (4.5,0) circle (.3) node[below=3pt]{$\frac{-9}{2}$};
\draw (5.5,0) node{$\cdots$};
\draw (0,-.5) to (0,.5);
\end{tikzpicture}
\end{center}
\end{example}

\begin{example} \label{ex:particles}
The following Maya diagram illustrates the state $S$ consisting of an
electron of energy $3/2$, and two positrons, of energy $1/2$ and
$5/2$.

\begin{center}
\begin{tikzpicture}[scale=.6]
\draw (0,-.5) to (0,.5);
\draw (-5.5,0) node{$\cdots$};
\draw (-4.5,0) circle (.3) node[below=3pt]{$\frac{9}{2}$};
\draw (-3.5,0) circle (.3) node[below=3pt]{$\frac{7}{2}$};
\draw (-2.5,0) circle (.3) node[below=3pt]{$\frac{5}{2}$};
\filldraw (-1.5,0) circle (.3) node[below=3pt]{$\frac{3}{2}$};
\draw (-.5,0) circle (.3) node[below=3pt]{$\frac{1}{2}$};
\draw (.5,0) circle (.3) node[below=3pt]{$\frac{-1}{2}$};
\filldraw (1.5,0) circle (.3) node[below=3pt]{$\frac{-3}{2}$};
\draw (2.5,0) circle (.3) node[below=3pt]{$\frac{-5}{2}$};
\filldraw (3.5,0) circle (.3) node[below=3pt]{$\frac{-7}{2}$};
\filldraw (4.5,0) circle (.3) node[below=3pt]{$\frac{-9}{2}$};
\draw (5.5,0) node{$\cdots$};
\end{tikzpicture}
\end{center}
\end{example}

\subsection{Paths}

We now describe a bijection between the set of partitions $\mathcal{P}$ to the
set of charge 0 states, that sends a partition
$\lambda\in\mathcal{P}_n$ of size $n$ to a state $S_\lambda$ with
energy $e(S_\lambda)=n$.  This bijection can be understood in two ways:
as recording the boundary path of $\lambda$, or recording the modified
Frobenius coordinates of $\lambda$.

\subsubsection{}
We draw partitions in ``Russian
notation'' -- rotated $\pi/4$ radians counterclockwise and scaled up
by a factor of $\sqrt{2}$, so that each segment of the border path of
$\lambda$ is centered above a half integer on the $x$-axis.  We traverse the boundary path of $\Lambda$ from left to right.  For each segment of the border path, we place an electron in the corresponding energy level if that segment of the border slopes up, and we leave the energy state empty if that segment of border path slopes down.

\begin{example} \label{ex:electronstopartitions}

We illustrate the bijection in the case of $\lambda=3+2+2$.  The corresponding state $S_\lambda$ consists of two
electrons with energy $5/2$ and $1/2$, and two positrons with energy
$3/2$ and $5/2$.

\begin{center}
\begin{tikzpicture}
\begin{scope}[gray, very thin, scale=.6]
\clip (-5.5, 5.5) rectangle (5.5, -5);
\draw[rotate=45, scale=1.412] (0,0) grid (6,6);
\end{scope}

\begin{scope}[rotate=45, very thick, scale=.6*1.412]
\draw (0,5.5) -- (0, 3) -- (1,3) -- (1,2) -- (3,2) -- (3,0) -- (5.5,0);
\end{scope}

\begin{scope}[scale=.6, dotted]

\draw (-4.5,0) -- (-4.5, 4.5);
\draw (-3.5,0) -- (-3.5, 3.5);
\draw (-2.5,0) -- (-2.5, 3.5);
\draw (-1.5,0) -- (-1.5, 3.5);
\draw (-.5,0) -- (-.5, 3.5);
\draw (.5,0) -- (.5, 4.5);
\draw (1.5,0) -- (1.5, 4.5);
\draw (2.5,0) -- (2.5, 3.5);
\draw (3.5,0) -- (3.5, 3.5);
\draw (4.5,0) -- (4.5, 4.5);
\end{scope}

\begin{scope}[scale=.6, yshift=-.5cm]
\draw[solid] (0,-.5) -- (0,.5);
\draw (-5.5,0) node{$\cdots$};
\draw (-4.5,0) circle (.3) node[below=3pt]{$\frac{9}{2}$};
\draw (-3.5,0) circle (.3) node[below=3pt]{$\frac{7}{2}$};
\filldraw (-2.5,0) circle (.3) node[below=3pt]{$\frac{5}{2}$};
\draw (-1.5,0) circle (.3) node[below=3pt]{$\frac{3}{2}$};
\filldraw (-.5,0) circle (.3) node[below=3pt]{$\frac{1}{2}$};
\filldraw (.5,0) circle (.3) node[below=3pt]{$\frac{-1}{2}$};
\draw (1.5,0) circle (.3) node[below=3pt]{$\frac{-3}{2}$};
\draw (2.5,0) circle (.3) node[below=3pt]{$\frac{-5}{2}$};
\filldraw (3.5,0) circle (.3) node[below=3pt]{$\frac{-7}{2}$};
\filldraw (4.5,0) circle (.3) node[below=3pt]{$\frac{-9}{2}$};
\draw (5.5,0) node{$\cdots$};
\end{scope}
\end{tikzpicture}
\end{center}

\end{example}

\subsubsection{Frobenius Coordinates}

The energies of the electrons and the positrons of $\lambda$ are
the \emph{modified Frobenius coordinates},

The $y$-axis dissects the partition $\lambda$ into two pieces.  The left side of $\lambda$ consists of $c$ rows, where $c$ is the number
of electrons.  The length of the $i$th row is the energy of
the $i$th electron.  The right half of $\lambda$ also consists of $c$ rows, with lengths the energies of the positrons.

\begin{example} Consider Example \ref{ex:electronstopartitions}.  If the $y$-axis was drawn in, left of the $y$-axis would be two rows, the bottom row having length 2.5 and the top row length .5 -- these were precisely the energies of the electrons in $S$.  Similarly, the right hand side has two rows of length 2.5 and 1.5, the energies of the positrons in $S$.

\end{example}

\subsubsection{Non-zero charge}

The bijection between partitions and states of charge zero may be
modified to give a bijection between partitions and states of charge $c$ for any $c\in\Z$.   Simply translate the partition to the right by $c$.

\subsection{Abaci}

Rather than view the Maya diagram as a series of stones in a line, we
now view it as beads on the runner of an abacus.  Sliding the beads
to be right justified allows the charge of the state to be read off,
as it is easy to see how many electrons have been added or are missing
from the vacuum state.

In what follows, we mix our metaphors and talk about electrons and protons on runners of an abacus.

\begin{example} \label{ex:mayabijection}
Consider Example \ref{ex:particles}, where the Maya diagram consists of
two positrons and an electron.  Pushing the beads to be right
justified, we see the first bead is one step to the right of zero, and
hence the original state had charge 1.

\begin{center}
\begin{tikzpicture}[scale=.6]
\draw (0,-.5) to (0,.5);
\draw (-5.5,0) node{$\cdots$};
\draw (-4.5,0) circle (.3) node[below=3pt]{$\frac{9}{2}$};
\draw (-3.5,0) circle (.3) node[below=3pt]{$\frac{7}{2}$};
\draw (-2.5,0) circle (.3) node[below=3pt]{$\frac{5}{2}$};
\filldraw (-1.5,0) circle (.3) node[below=3pt]{$\frac{3}{2}$};
\draw (-.5,0) circle (.3) node[below=3pt]{$\frac{1}{2}$};
\draw (.5,0) circle (.3) node[below=3pt]{$\frac{-1}{2}$};
\filldraw (1.5,0) circle (.3) node[below=3pt]{$\frac{-3}{2}$};
\draw (2.5,0) circle (.3) node[below=3pt]{$\frac{-5}{2}$};
\filldraw (3.5,0) circle (.3) node[below=3pt]{$\frac{-7}{2}$};
\filldraw (4.5,0) circle (.3) node[below=3pt]{$\frac{-9}{2}$};
\draw (5.5,0) node{$\cdots$};

\begin{scope}[yshift=-2.5cm]
\draw[-triangle 45, snake=snake,line after snake=1mm] (-4,0)--(4,0) node [above,  text centered,midway] {Push beads};

\end{scope}

\begin{scope}[yshift=-4cm]
\draw (0,-.5) to (0,.5);
\draw (-5.5,0) node{$\cdots$};
\draw (-4.5,0) circle (.3) node[below=3pt]{$\frac{9}{2}$};
\draw (-3.5,0) circle (.3) node[below=3pt]{$\frac{7}{2}$};
\draw (-2.5,0) circle (.3) node[below=3pt]{$\frac{5}{2}$};
\draw (-1.5,0) circle (.3) node[below=3pt]{$\frac{3}{2}$};
\draw (-.5,0) circle (.3) node[below=3pt]{$\frac{1}{2}$};
\draw (.5,0) circle (.3) node[below=3pt]{$\frac{-1}{2}$};
\filldraw (1.5,0) circle (.3) node[below=3pt]{$\frac{-3}{2}$};
\filldraw (2.5,0) circle (.3) node[below=3pt]{$\frac{-5}{2}$};
\filldraw (3.5,0) circle (.3) node[below=3pt]{$\frac{-7}{2}$};
\filldraw (4.5,0) circle (.3) node[below=3pt]{$\frac{-9}{2}$};
\draw (5.5,0) node{$\cdots$};
\end{scope}

\end{tikzpicture}
\end{center}
\end{example}

\subsubsection{Cells and hook lengths}

The cells $\square\in\lambda$ are in bijection with the
\emph{inversions} of the boundary path; that is, by pairs of segments
$(\text{step}_1, \text{step}_2)$, where $\text{step}_1$ occurs before $\text{step}_2$,
but $\text{step}_1$ is traveling NE and $\text{step}_2$ is traveing SE.  The bijection
sends $\square$ to the segments at the end of its arm and leg.

Translating to the fermionic viewpoint, cells of $\lambda$ are in
bijection with pairs 
$$\left\{(e, e-k)\big | e\in\Z_{1/2}, k>0\right\}$$ of a filled energy level $e$ and an
empty energy level $e-k$ of lower energy; we call such a pair an \emph{inversion}.  The hook length
$h(\square)$ of the corresponding cell is $k$.

If $(e,e-k)$ is such a pair, reducing the energy of the electron from
$e$ to $e-k$ changes $\lambda$ by removing the
rim hook corresponding to the cell $\square$.  This rim-hook has
length $k$.

\begin{example}
The cell $\square=(2,1)$ of $\lambda=3+2+2$ (See Example \ref{ex:electronstopartitions}).
Here, $h(\square)=3$, and corresponds to the electron in energy state $1/2$
and the empty energy level $-5/2$; which are three apart.

\end{example}

\subsection{Bijections}

Rather than place the electrons corresponding to $\lambda$ on one
runner, place them on $a$ different runners, putting the energy
levels $ka-i-1/2$ on runner $i$.

If the hooklength $h(\square)=ka$ is divisible by $a$, then the two
energy levels of $\textrm{inversion}(\square)$ lie on the same
runner.  Similarly, any inversion of energy states on the same
runner corresponds to a cell with hook length divisible by $a$.

Thus, $\lambda$ is an $a$-core if and only if the beads on each runner of the $a$-abacus are right justified.  Although the total charge of all the runners must be zero, the charge need not be evenly divided among the runners.  Let
$c_i$ be the charge on the $i$th runner; then we have $\sum c_i=0$, and the $c_i$ determine $\lambda$.

Similarly, given any $\mathbf{c}=(c_0,\dots,c_{a-1})\in\Z^a$ with $\sum c_i=0$, there is a unique right justified abacus with charge  $c_i$ on the $i$th runner.  The coresponding partition is an $a$-core which we denote $\core_a(\mathbf{c})$.

We have shown:

\begin{lemma}
There is a bijection $$\core_a:\{(c_0,\dots,c_{a-1}|c_i\in\Z, \sum c_i=0\}\to \{\lambda | \lambda \text{ is in $a$-core} \}$$
\end{lemma}

\begin{example}
We illustrate that $\core_3(0,3, -3)=7+5+3+3+2+2+1+1$.

\begin{center}
\begin{tikzpicture}

\begin{scope}[gray, very thin, scale=.4]
\clip (-11.5, 11.5) rectangle (11.5, -10);
\draw[rotate=45, scale=1.412] (0,0) grid (12,12);
\end{scope}

\begin{scope}[rotate=45, ultra thick, scale=.4*1.412]
\begin{scope}[red]
\draw (0,11.5)--(0,11);
\draw (0,9)--(0,8);
\draw (1,7)--(1,6);
\draw (2,5)--(2,4);
\draw (3,3)--(4,3);
\draw (5,2)--(6,2);
\draw (7,1)--(8,1);
\draw (9,0)--(10,0);
\end{scope}
\begin{scope}[blue]
\draw (0,11)--(0,10);
\draw (0,8)--(0,7);
\draw (1,6)--(1,5);
\draw (2,4)--(2,3);
\draw (4,3)--(4,2);
\draw (6,2)--(6,1);
\draw (8,1)--(8,0);
\draw (10,0)--(11,0);
\end{scope}
\begin{scope}[green]
\draw (0,10)--(0,9);
\draw (0,7)--(1,7);
\draw (1,5)--(2,5);
\draw (2,3)--(3,3);
\draw (4,2)--(5,2);
\draw (6,1) -- (7,1);
\draw (8,0)--(9,0);
\draw (11,0)--(11.5,0);
\end{scope}

\end{scope}

\begin{scope}[scale=.4, yshift=-.5cm, red]
\draw(-11.5,0) circle(.3);
\draw(-8.5, 0)  circle (.3);
\draw(-5.5,0)  circle (.3);
\draw(-2.5,0)  circle (.3);
\filldraw(.5,0)  circle (.3);
\filldraw(3.5,0)  circle (.3);
\filldraw(6.5,0)  circle (.3);
\filldraw(9.5,0)  circle (.3);
\end{scope}

\begin{scope}[scale=.4, yshift=-1.5cm, blue]
\draw(-10.5,0) circle(.3);
\draw(-7.5, 0)  circle (.3);
\draw(-4.5,0)  circle (.3);
\draw(-1.5,0)  circle (.3);
\draw(1.5,0)  circle (.3);
\draw(4.5,0)  circle (.3);
\draw(7.5,0)  circle (.3);
\filldraw(10.5,0)  circle (.3);
\end{scope}

\begin{scope}[scale=.4, yshift=-2.5cm, green]
\draw(-9.5,0) circle(.3);
\filldraw(-6.5, 0)  circle (.3);
\filldraw(-3.5,0)  circle (.3);
\filldraw(-.5,0)  circle (.3);
\filldraw(2.5,0)  circle (.3);
\filldraw(5.5,0)  circle (.3);
\filldraw(8.5,0)  circle (.3);
\filldraw(11.5,0)  circle (.3);
\end{scope}

\begin{scope}[scale=.4]
\draw (0,-3)--(0,0);
\end{scope}

\end{tikzpicture}
\end{center}
\end{example}

\subsection{Size of an \texorpdfstring{$a$}{a}-core}

\begin{theorem} \label{thm:quadratic}
$$|\core_a(\mathbf{c})|=\frac{a}{2}\sum_{k=0}^{a-1} c_k^2+ kc_k$$
\end{theorem}

\begin{proof}
 If  $c_k>0$ the $k$th runner has $c_k$ positrons, with
 energies 
\begin{align*}
(k+1/2)&,\\
(k+1/2)&+a, \\
(k+1/2)&+2a,\\
\vdots\quad &\quad \vdots \\
 (k+1/2)&+(c_k-1)a
\end{align*}
 and so the
 particles on the $k$th runner have total energy
 $$\frac{a}{2}(c_k^2-c_k)+(k+1/2)c_k.$$

 If $c_k<0$, the $k$th runner has $-c_k$ electrons, and a similar calculation shows they have a total energy of $$\frac{a}{2}(c_k^2+c_k)-c_k(a-k-1/2)=\frac{a}{2}(c_k^2-c_k)+(k+1/2)c_k.$$

Since $\sum c_k=0$, the total energy of all particles simplifies to $\frac{a}{2}\sum (c_k^2+kc_k)$.
\end{proof}

\section{Simultaneous Cores} \label{sec:numberandsize}

We now turn to studying the set of $b$-cores within the lattice
$\Lambda_a$ of $a$-cores.
\subsection{\texorpdfstring{$(a,b)$}{(a,b)}-cores form a simplex}

First, some notation and conventions.

Let $r_a(x)$ be the remainder when $x$ is divided by $a$, and $q_a(x)$ to be the integer part of $x/a$, so that $x=aq_a(x)+r_a(x)$ for all $x$.
Furthermore, we use cyclic indexing for $\mathbf{c}\in\Lambda_a$;
that is, for $k\in\Z$, we set $c_k=c_{r_a(k)}$.

\begin{lemma}
Within the lattice of $a$ cores, the set of $b$ cores are the
lattice points satisfying the inequalities $$c_{i+b}-c_{i}\leq q_a(b+i)$$
for $i\in\{0,\dots,a-1\}$.
\end{lemma}

\begin{proof}

Fix $\mathbf{c}\in\Lambda_a$, and consider the corresponding
$a$-abacus.

Let $\lambda=\core_a(\mathbf{c})$ be an $a$ core, and let $e_i$ denote the energy of the highest electron the $i$th runner.  We claim that $\core_a(\mathbf{c})$ is a $b$-core if an only if for each $i$, the energy state $e_i-b$ is filled.

Certainly this condition is necessary.  To see that it is sufficient, suppose that $\lambda$ is an $a$-core, and that $e_i-b$ are all filled.  To see $\lambda$ is a $b$ core, we must show that for any filled energy level $L$, that $L-b$ is filled.  

Suppose that $L$ is on the $i$th runner; then $L=e_i-aw$ for some $w\geq 0$, and so $L-b=(e_i-b)-aw$.  But by supposition $e_i-b$ is a filled state, and $e_i-b-aw$ is to the right of it and on the same runner, and so it must be filled since $\lambda$ is an $a$-core.

Now, the energy state $e_{i}-b$ is on runner $r_a(i+b)$, and so
$\lambda$ is $b$-core if and only if $e_{i}-b\leq e_{i+b}$ (recall that we are using cyclic indexing).

Substituting $e_k=-ac_k-r(k)-1/2$ and simplifying gives that our inequality is equivalent to
\begin{align*}
a(c_{i+b}-c_i)& \leq b+i-r_a(i+b)
\end{align*}
and hence to
$$c_{i+b}-c_i \leq q_a(b+i).$$

\end{proof}

We have $a$ hyperplanes in an $a-1$ dimensional space; they either
form a simplex or an unbounded polytope.

\begin{remark}  The same analysis sheds light on the case when $a$ and
  $b$ are not relatively prime, which has been studied in \cite{AKS}.

Let $d=\gcd(a,b)$; then any $d$-core is also an $(a,b)$-core, and so
there are no longer finitely many $(a,b)$-cores.

The inequalities given for $\SC_a(b)$ still describe the space of
$(a,b)$-cores when $a,b$ are no longer relatively prime, but these
inequalities no longer describe a simplex.  The inequalities no longer
relate all the $c_i$ to each other; rather, they decouple into $d$ sets of $a/d$ of variables

\begin{align*}
S_0 & =\{c_0, c_d, c_{2d},\dots, c_{a-d}\} \\
 S_1& =\{c_1,c_{d+1},\dots, c_{a-d+1}\} \\
& \quad \quad \cdots \\
S_{d-1}&=\{c_{d-1}, c_{2d-1},\dots, c_{a-1}\} \\
\end{align*}
The charges $c_i$ in a given group must be close together, but for any
vector $(v_0,\dots, v_{d-1})$ with $\sum v_i=0$, we may shift each
element of $S_i$ by $v_i$ and all inequalities will still be satisfied.

These shifts generate a lattice, and the remaining choices of the $c_i$ in each group are related to each other is a polytope, and so the set of $(a,b)$ cores is a finite number of translates of a lattice.  The sum over $Q$ over the points in a lattice will be a theta function, and so we see the generating function of $(a,b)$ cores will be a finite sum of theta functions, and hence modular.

\end{remark}






\subsubsection{Coordinate shift}
In the charge coordinates $\mathbf{c}$, neither the hyperplanes defining the set of $b$ cores nor the quadratic form $Q$ are symmetrical about the origin.  We shift coordintaes to remedy this.
\begin{definition}
Define $\mathbf{s}=(s_0,\dots, s_a)\in V_{a-1}$ by
$$s_i=\frac{i}{a}-\frac{a-1}{2a}$$
\end{definition}

The $i/a$ term ensures $s_{i+1}-s_i=1/a$; subtracting $\frac{a-1}{2a}$ ensures that  $\mathbf{s}\in V_a$, i.e. $\sum s_i=0$.
 
\begin{lemma}
In the shifted charge coordinates
$$x_i=c_i+s_i$$
the inequalities defining the set of $b$ cores become
$$x_{i+b}-x_{i}\leq b/a$$
and the size of an $a$-core is given by
$$Q(\mathbf{x})=-\frac{a^2-1}{24}+\frac{a}{2}\sum_{i=0}^{a-1} x_i^2$$
\end{lemma}

\begin{proof}
That the linear term of $Q$ vanishes in the $\mathbf{x}$ coordinates follows immediately from the definition of $\mathbf{s}$.  The constant term of $Q$ in the $\mathbf{x}$ coordinates is $-\frac{a}{2}\sum_{i=0}^{a-1} s_i^2$, which a short computation shows is $-\frac{a^2-1}{24}$.

The statement about the set of $b$-cores follows from the computation

\begin{align*}
x_{i+b}-x_{i}&=c_{i+b}-c_i+s_{i+b}-s_i \\
&\leq  q_a(i+b)+r_a(i+b)/a-i/a \\
&=(b+i)/a-i/a \\
&=b/a
\end{align*}
\end{proof}

Although we often use the $x$ coordinates, to show that the simplex of $\SC_a(b)$ is isomorphic to the simplex $\TD_a(b)$ of trivial determinant representations another change of variables is needed.
 
\begin{lemma} \label{lem:standardsimplex}
Let $a$ and $b$ be relatively prime, and let 
$$k=-\frac{b+1}{2}\pmod a$$
  Then the change of variables
$$z_i=x_{ib+k}-x_{(i+1)b+k}+b/a$$
gives an isomorphism between the rational simplices $\SC_a(b)$ and $\TD_a(b)$. 
\end{lemma}

\begin{proof}
It is immediate that the $z_i$ satisfy $\sum z_i=b$ and $z_i\geq 0$.  The integrality of the $z_i$ follows from the fact that the fractional part of $x_i-x_j$ is $(i-j)/a$.  We must show $\sum iz_i=0\mod a$.

One computes:
$$\sum_{i=0}^{a-1} iz_i=-ax_k+\sum_{i=0}^{a-1} x_i+\frac{b}{a}\sum_{i=0}^{a-1} i$$
Since the fractional part of $x_k$ is $s_k=k/a-(a-1)/2a$, plugging in the definition of $k$ gives that $ax_k=-b/2\pmod a$.  Since $\sum x_i=0$ and $\sum i=(a-1)a/2$, we see $\sum iz_i=0\pmod a$.

A further computation shows this change of variables is invertible.
\end{proof}

\begin{corollary}[Anderson \cite{anderson}] \label{cor:anderson} The number of simultaneous $(a,b)$-cores is $\Cat_{a,b}$.
\end{corollary}

\begin{proof} This follows quickly from Lemma \ref{lem:standardsimplex}.

The scaled simplex $b\Delta_a$ has $\binom{a+b-1}{a-1}$
  usual lattice points.  Cyclicly permuting the variables preserves
  $b\Delta_a$ and the standard lattice, and when $b$ is relatively prime to $a$ it cyclicly permutes the $a$ cosets of the charge lattice.   

Thus the standard lattice points in
  $b\Delta_a$ are equidistributed among the $a$-cosets of the
  charge lattice, and hence each one contains $\frac{1}{a}\binom{a+b-1}{a-1}=\Cat_{a,b}$.

\end{proof}

\subsection{The size of simultaneous cores}

We now have all the ingredients needed to prove Armstrong's conjecture.  We derive it as a consequence of:

\begin{theorem} \label{thm:polynomial}
For fixed $a$, and $b$ relatively prime to $a$,  the average size of an $(a,b)$-core is a polynomial of degree 2 in $b$.
\end{theorem}

\begin{proof}
For fixed $a$, the number of $a$-cores is $1/a$ times the
number of lattice points in $b\Delta_{a-1}$, which is a polynomial
$F_a(b)$ of degree $a-1$.  In the $x$-coordinates $Q=|\core_a|$
is invariant under $S_a$, and in particular rotation, so the sum of the sizes of all $(a,b)$-cores is $1/a$ times the sum of $Q$ over the lattice points in $b\Delta_{a-1}$.   By Euler-Maclaurin theory, the number of points in $b\Delta_{a-1}$ is a polynomial $G_a(b)$ of degree $a+1$.

Thus, the average value of an $(a,b)$-core is $G_a(b)/F_a(b)$, the quotient of a polynomial of degree $a+1$ in $b$ by a polynomial of degree $a-1$ in $b$.  To show this is a polynomial of degree two in $b$, we need to show that every root of $F_a$ is a root of $G_a$.

Corollary \ref{cor:anderson} says that the roots of $F_a$ are $-1,
-2,\dots, -(a-1)$.  We now give another derivation of this fact, using Ehrhart reciprocity, that easily adapts to shown these are also roots of $G_a$. 

Ehrhart reciprocity says that $F_a(-x)$ is, up to a sign, the number of
points in the \emph{interior} of $x\Delta_{a-1}$.  The interior
consists of the points in $x\Delta_{a-1}$ none of whose coordinates
are zero, and so the first interior point in $x\Delta_{a-1}$ is  $(1,1,\dots,1)\in a\Delta_{a-1}$.   Thus, $F_a(b)$ vanishes at $b=-1,\dots, -(a-1)$, and as it has degree $a-1$ it has no other roots.

Ehrhart reciprocity extends to Euler-Maclaurin theory, to say that up to a sign $Q_a(-x)$ is the sum of $F$ of the interior points of
$x\Delta_{a-1}$. Thus $Q_a(-x)$ also vanishes at
$b=-1,\dots,-(a-1)$, and so $P_a/Q_a$ is a polynomial of degree 2.

\end{proof}

\begin{corollary} \label{cor:averagesize}
When $(a,b)$ are relatively prime, the average size of an $(a,b)$ core is $(a+b+1)(a-1)(b-1)/24$
\end{corollary}

\begin{proof}

Fix $a$, and let $P_a(b)=G_a(b)/F_a(b)$ be the degree two polynomial that
gives the average value of the $(a,b)$-cores when $a$ and $b$ are
relatively prime.  As we know $P_a(b)$ is a polynomial of degree 2, we can determine it by computing only three values.

First, we find the two roots of $P_a(b)$.  As the only $1$ core is the empty partition, we have $F_a(1)=1$ and $G_a(1)=0$, and so $P_a(1)=0$.

Ehrhart reciprocity gives that $G_a(-a-b)$ is, up to a sign, the sum of $Q$ over the lattice points in the interior of $(a+b)\Delta_a$, which are just the lattice points contained in $b\Delta_a$, and hence equal to $G_a(b)$.  In particular, $P_a(-a-1)=0$.  

Finally, we compute $P_a(0)$.  It is clear that $\mathcal{S}_a(0)=\{0\}$.  Although this is not a point of $\Lambda_a$, it is in $\Lambda_a^\prime$, and so $P_a(0)=Q(0)=-(a^2-1)/24$.
\end{proof}

\subsection{Self-conjugate \texorpdfstring{$(a,b)$}{(a,b)}-cores} \label{sec:conjugate}

In Lemma \ref{lem:conjugatedual}, we show that under the bijection between $(a,b)$-cores and $b$-dimensional representations of $\Z_a$ with trivial determinant, conjugating a partition corresponds to sending a representation $V$ to its dual $V^*$.  In the lattice point of view, this is a linear map $T$, and hence the self-dual $(a,b)$-cores correspond to the lattice points in the fixed point locus of $T$.  

We show in Lemma \ref{lem:conjugatecount}, that the $T$-fixed lattice points in $\SC_a(b)$ are the lattice points in the $\lfloor a/2\rfloor$ dimensional simplex $\lfloor b/2\rfloor \Delta_{\lfloor a/2\rfloor}$, hence rederiving the count of simultaneous $(a,b)$-core partitions.  

Once we have done this, an analogous application of Euler-Maclaurin theory reproves the statement about the average value.

Let $T:V_a\to V_a$ be the linear map given by
$$T(c_i)=-c_{-1-i}$$
It is easy to check that when translated to core partitions, $T$ corresponds to taking the conjugate, that is: $$\core_a(c)^T=\core_a(T(c))$$
Thus the set of self-conjugate $(a,b)$-cores is the $T$ fixed locus of $\SC_a(b)$.  

Since $T(s)=s$, the same formula holds in the shifted coordinates $\bf{x}$.

\begin{lemma} \label{lem:conjugatedual}
Under the isomorphism between $\SC_a(b)$ and $\TD_a(b)$ established in Lemma \ref{lem:standardsimplex}, the map $T$ sending a partition to its conjugate corresponds to taking the dual $\Z_a$ representation.
\end{lemma}

\begin{proof}
We want to show $T(z_i)=z_{-i}$.  We compute:
\begin{align*}
T(z_i)&=T(x_{ib+k}-x_{(i+1)b+k}) \\
&=-x_{-ib-k-1}+x_{-ib-b-k-1} \\
& =x_{-ib+k-(b+1+2k)}-x_{(-i+1)b+k-(b+1+2k)}
\end{align*}
And so we need $b+1+2k=0\pmod a$, but this is exactly the definition of $k$ in Lemma \ref{lem:standardsimplex}.
\end{proof}

\begin{lemma} \label{lem:conjugatecount}
The number of $b$-dimensional, self-dual $\Z_a$ representations with trivial determinant is $$\binom{\big\lfloor\frac{a}{2}\big\rfloor+\big\lfloor\frac{b}{2}\big\rfloor}{\big\lfloor\frac{a}{2}\big\rfloor}$$
\end{lemma}

\begin{proof}
  Let $a=2k$ or $2k+1$.  We give a bijection between the representations in question and $k$-tuples of non-negative integers $(z_1,\dots, z_k)$ with $2\sum z_i\leq b$.  The set of such $z_i$ are the lattice points in $\lfloor b/2\rfloor\Delta_{\lfloor a/2\rfloor}$, which are counted by the given binomial coefficient.

First, suppose that $a=2k+1$.  Then the only irreducible self-conjugate representation is the identity, and $T$ has a $k$ dimensional fixed point set consisting of points of the form $(u_0, u_1,\dots, u_k,u_k,\dots, u_1)$.  Thus, $\sum_{i=1}^k 2u_k\leq b$, and value of $u_0$ is fixed by $u_0+2\sum_{i=1}^k u_i=b$.

When $a=2k$, there are two irreducible self-conjugate representations, the identity and the sign representation induced by the surjection $\Z_a\to\Z_2$.  Again, $T$ has a $k$ dimensional fixed point set, this time consisting of points of the form $(u_0, u_1,\dots u_{k-1}, w_{k}, u_{k-1},\dots,u_1)$.  Now for such a representation, having trivial determinant is equivalent to $w_k$ being even, say $w_k=2u_k$.  Then again we have $\sum_{i=1}^k 2u_k\leq b$, with $u_0$ being determined by $u_0+2\sum_{i=1}^{k}u_i=b$. 
\end{proof}

\begin{proposition} \label{prop:conjugatesize}
Let $a$ and $b$ be relatively prime.  Then the average size of a self-conjugate $(a,b)$-core is $(a-1)(b-1)(a+b+1)/24$.
\end{proposition}

\begin{proof}

Since $a$ and $b$ are relatively prime, at most one is even, so we may assume $a$ is odd.  

The proof is essentially the same as that for all $(a,b)$-cores.  One complication is that it seems we must treat odd and even values of $b$ separately.   In each each case, an argument identical to Lemma \ref{thm:polynomial} gives that the average size is a polynomial of degree 2 in $b$.  A priori, we may have different polynomials for $b$ odd and $b$ even; however, the symmetry $(a,b)\leftrightarrow (a, -a-b)$ coming from Ehrhart reciprocity still holds and interchanges odd and even values of $b$, and so if we can compute three values of either polynomial (that don't get identified by this symmetry), we identify both polynomials.

All $1$ and $2$ cores are self conjugate, and thus if $b$ is 1 or 2, the average value is the same.  The arguments made in Corollary \ref{cor:averagesize} for the value of the polynomial at $b=0$ holds for self-conjugate partitions as well, giving a third value.
\end{proof}

\section{Toward \texorpdfstring{$q$}{q}-analogs} \label{sec:qcat}
In this section, we apply our lattice point and simplex point of view on simultaneous cores to the $q$-analog of rational Catalan numbers; the next section approaches $(q,t)$-analogs.

\subsection{\texorpdfstring{$q$}{q}-numbers}

Recall the standard $q$ analogs of $n$, $n!$ and $\binom{n}{k}$:
\begin{align*}
[n]_q&=1+q+q^2+\cdots+q^{n-1}=\frac{1-q^n}{1-q} \\
[n]_q!&=[n]_q[n-1]_q\cdots [2]_q[1]_q \\
{n \brack k}_q&=\frac{[n]_q!}{[k]_q![n-k]_q!}
\end{align*}
These three functions are polynomials with positive integer coefficients, i.e., they are elements of $\N[q]$.

The $q$ rational Catalan numbers are given by the obvious formula:
\begin{definition}
$$\Cat_{a,b}(q)=\frac{1}{[a+b]_q} {a+b \brack a}_q=
\frac{(1-q^{b+1})(1-q^{b+2})\cdots (1-q^{b+a-1})}{(1-q^2)(1-q^3)\cdots (1-q^a)}$$
\end{definition}

\subsection{Graded vector spaces}

One place $q$ analogs occur naturally is in graded vector spaces.
\begin{definition}
If $V$ is a graded vector space, with $V_k$ denoting the weight $k$ subspace of $V$, we define $$\dim_qV=\sum_{k\in \N} q^k\dim V_k.$$
\end{definition}

\begin{proposition}
Let $p_i$ be a variable of weight $i$, then $\C[p_1,\dots, p_n]$ has
finite dimensional graded pieces, and
$$\dim_q \C[p_1,\dots, p_n]=\frac{1}{(1-q)(1-q^2)\cdots(1-q^n)}$$

If $V$ is a vector space with $\dim_qV=[n]_q$, then
$$\dim_q \Sym^b V={n+b-1 \brack n-1}_q$$
\end{proposition}

These statements can be interpreted geometrically in terms of lattice
points   The monomials in $\C[p_1,\dots, p_n]$ correspond to the lattice points in an $n$ dimensional unimodular cone; the monomials in $\Sym^b V$ correspond to lattice points in the scaled standard simplex $b\Delta_{a-1}$; the $q$-analogs of the statements listed above are $q$ counting the lattice points, where the weights of the $i$th primitive lattice vector on the ray of the cone has weight $q^i$.

\begin{example}

The following diagram illustrates ${b+a-1 \brack a-1}_q$ as $q$-counting
  standard lattice points in $b\Delta_{a-1}$ in the case $a=3$ and
  $b=4$.   Letting $b$ go to infinity corresponds to extending
  the arrows and the lattice points between them infinitely far to the
  upper right, showing that $\prod_{k=1}^{a-1} \frac{1}{1-q^k}$
  $q$-counts the points in a standard cone.
\begin{center}
\begin{tikzpicture}

\filldraw (0,0)  circle(.03) node[above] {$1$};
\filldraw (1,0) circle(.03) node[above] {$q$};
\filldraw (2,0) circle(.03) node[above] {$q^2$};
\filldraw (3,0) circle(.03) node[above] {$q^3$};
\filldraw (4,0) circle(.03) node[above] {$q^4$};

\filldraw (.5, 1.732058/2) circle(.03) node[above] {$q^2$};
\filldraw (1.5, 1.732058/2) circle(.03) node[above] {$q^3$};
\filldraw (2.5, 1.732058/2) circle(.03) node[above] {$q^4$};
\filldraw (3.5, 1.732058/2) circle(.03) node[above] {$q^5$};

\filldraw (1, 1.732058) circle(.03) node[above] {$q^4$};
\filldraw (2, 1.732058) circle(.03) node[above] {$q^5$};
\filldraw (3, 1.732058) circle(.03) node[above] {$q^6$};

\filldraw (1.5, 1.5*1.732058) circle(.03) node[above] {$q^6$};
\filldraw (2.5, 1.5*1.732058) circle(.03) node[above] {$q^7$};

\filldraw (2, 2*1.732058) circle(.03) node[above] {$q^8$};

\draw[-triangle 45] (0,-.5)-- node[below] {$\cdot q$} (4, -.5);
\draw[-triangle 45] (-.5, .25*1.732058)--node[above]{$\cdot
  q^2$}  (1.5, 2.25*1.732058);

\end{tikzpicture}
\end{center}
\end{example}

\subsection{A \texorpdfstring{$q$}{q}-version of cone decompositions}

\subsubsection{Lawrence Varchenko}
Recall the decomposition of a simplicial polytope $\mathcal{P}$ in a vector space $V$ of dimension $n$ as a signed sum of cones based at their vertices, called the Lawrence-Varchenko decomposition:

\begin{center}
  \begin{tikzpicture}
  \draw[fill=blue!20] (0,0) node [below left] {$A$} --(1,0) node [below right] {$B$}--(2,1) node [above right] {$C$}--(0,1) node [above left] {$D$}--cycle;
  \draw[->] (0,0)-- node [above] {$v$} (1.1,.7);

\draw (2.5,.5) node {$=$};

\begin{scope}[xshift=3cm]
    \draw (0,0) node [below left] {$A$};
    \draw[-triangle 45] (0,0)-- (0,1);
    \draw[-triangle 45] (0,0)-- (1, 0);
\end{scope}

\begin{scope}[xshift=5cm]
    \draw (-.6,.5) node {$-$};
    \draw (0,0) node [below left] {$B$};
    \draw[-triangle 45] (0,0)-- (1,0);
    \draw[-triangle 45] (0,0)-- (1, 1);
\end{scope}

\begin{scope}[xshift=7cm]
    \draw (-.6,.5) node {$+$};
    \draw (0,0) node [below left] {$C$};
    \draw[-triangle 45] (0,0)-- (1,0);
    \draw[-triangle 45] (0,0)-- (1, 1);
\end{scope}

\begin{scope}[xshift=9cm]
    \draw (-.6,.5) node {$-$};
    \draw (0,0) node [below left] {$D$};
    \draw[-triangle 45] (0,0)-- (1,0);
    \draw[-triangle 45] (0,0)-- (0, 1);
\end{scope}

  \end{tikzpicture}

\end{center}

First, pick a generic direction vector $v\in V$.
At each vertex $v_i$, $n$ facets of $\mathcal{P}$ meet; if we extend these
facets to hyperplanes, they cut $V$ into orthants.  Let
$\mathcal{C}_i$ be the orthant at $v_i$ that contains our direction
vector $v$.  Let $f_i$ be the number of hyperplanes that must be crossed to get from $\mathcal{C}_i$ to $\mathcal{P}$.

Then:
$$\mathcal{S}=\sum_{i=0}^{k} (-1)^{f_i} \mathcal{C}_i$$

To deal correctly with the boundary of $P$, one must correctly include or exclude portions of the boundary of $\mathcal{C}_i$, but this subtlety won't matter to us.

\subsubsection{}

The algebraic structure of $\binom{b+a-1}{a-1}$ suggests  a refinement of the Lawrence-Varchenko decomposition of $b\Delta_{a-1}$ for $q$-counting the lattice points.

Expanding the numerator of $\binom{a+b-1}{a-1}$ as $(1-q^{b+1})\cdots (1-q^{b+a-1})$ there are $\binom{a-1}{k}$ terms obtained from choosing $1$ from $n-k$ factors and $q^M$ from $k$ factors. Each such term has sign $(-1)^k$, and the exponent of $q$ is slightly larger than $kb$.  We interpret these $\binom{a-1}{k}$ terms as making up the polarized tangent cone at the $k$th vertex.

The polarized tangent cone at the $k$th vertex $v_k$ does not carry the standard $q$-grading.  However, it appears the cone at $v_k$ may be subdivided into $\binom{a-1}{k}$ smaller cones that do have the standard $q$-grading, essentially by intersecting with the $A_{a-1}$ hyperplane arrangement translated to $v_k$.

\begin{example} We illustrate the decomposition of $b\Delta_2$ suggested by
$${b+2 \brack 2}_q=\frac{1}{(1-q)(1-q^2)}\left(1-q^{b+1}-q^{b+2}+q^{2b+3}\right)$$

\begin{tikzpicture}

\clip (-1.5, -2.5)--(-1.5, 8.5)--(8.5, 8.5)--(8.5, -2.5)--cycle;
\fill[blue!20] (0,0)--(3.5,0)--(3.5/2,3.5*1.7320508/2)--cycle;
\fill[red!20] (4,0)--(2,4*1.7320508/2)--(4.25,8.5*1.7320508/2)--(8.25, 8.5*1.732508/2)--cycle;
\fill[orange!40] (5,0)--(8.5,0)--(8.5,7*1.7320508/2)--cycle;
\fill[green!20]  (3/2,5*1.7320508/2)-- (3/2-3.5/2,8.5*1.7320508/2) --(3/2+3.5/2,8.5*1.7320508/2) --cycle;

\foreach \x in {-7,-6,...,10}{
    \foreach \y in {-7,-6,...,10}{
        \filldraw (\x+\y/2,1.7320508*\y/2) circle(.05);}}

\draw[dashed] (0,0)--(3.5,0);
\draw[dashed] (0,0)--(3.5/2,3.5*1.7320508/2);
\draw (1/2, 1.7320508/6) node {$+$};
\filldraw[blue] (0,0) circle(.05) node[below]{$1$};
\draw[-triangle 45] (0,0)--(.5,0) node[below]{$\cdot q$};
\draw[-triangle 45] (0,0)--(.25,1.7320508/4) node[left]{$\cdot q^2$};

\draw[dashed] (4,0)--(4-3.5/2,3.5*1.7320508/2);
\draw[dashed] (4,0)--(4+3.5/2,3.5*1.7320508/2);
\draw (4,1.7320508/3) node {$-$};

\draw[-triangle 45] (4,0)-- (4-.25,1.7320508/4) node[left]{$\cdot q$};
\draw[-triangle 45] (4,0)--(4.25,1.7320508/4) node[right]{$\cdot q^2$};

\filldraw (3,0) circle(.05) node[below]{$q^{b}$};
\filldraw[red] (4,0) circle(.05) node[below]{$q^{b+1}$};

\draw[dashed] (5,0)--(8.5,0);
\draw[dashed] (5,0)--(5+3.5/2,3.5*1.7320508/2);
\draw (5.5, 1.7320508/6) node {$-$};
\draw[dashed, -triangle 45] (5,0)--(5.5,0) node[below]{$\cdot q$};
\draw[dashed, -triangle 45] (5,0)--(5.25,1.7320508/4) node[above]{$\cdot q^2$};

\filldraw[orange] (5,0) circle(.05) node[below]{$q^{b+2}$};

\draw[dashed] (3/2,5*1.7320508/2)--(3/2-3.5/2,8.5*1.7320508/2);
\draw[dashed] (3/2,5*1.7320508/2)--(3/2+3.5/2,8.5*1.7320508/2);

\draw (3/2, 17*1.7320508/6) node {$+$};

\draw[dashed, -triangle 45] (3/2,5*1.7320508/2)--(1.25,5.5*1.7320508/2) node[left]{$\cdot q$};
\draw[dashed, -triangle 45] (3/2,5*1.7320508/2)--(1.75,5.5*1.7320508/2) node[right]{$\cdot q^2$};

\filldraw (3/2,3*1.7320508/2) circle(.05) node[left]{$q^{2b}$};

\filldraw[green] (3/2,5*1.7320508/2) circle(.05) node[below]{$q^{2b+3}$};

\end{tikzpicture}

\end{example}

Together with $\Cat_{a,b}=\dim_\C(\Sym^b\C[\Z_a])^{\Z_a}$, one
might hope that we could give $\C[\Z_a]$ a grading so that we have
$$\Cat_{a,b}(q)=\dim_q(\Sym^b \C[\Z_a])^{\Z_a}$$
This naive hope does not appear possible.  However, we now describe a
conjectural weakening of it.

\subsection{Sublattices and shifting}

We begin by rewriting $\Cat_{a,b}(q)$.  Since $[a]_{q^k}=(1-q^{ak})/(1-q^k)$, we have

\begin{align*}
\Cat_{a,b}(q)&=\frac{(1-q^{b+1})(1-q^{b+2})\cdots (1-q^{b+a-1})}{(1-q^2)\cdots(1-q^a)} \cdot \frac{[a]_{q^2}[a]_{q^3}\cdots[a]_{q^{a-1}}}{[a]_{q^2}[a]_{q^3}\cdots[a]_{q^{a-1}}} \\
&=\frac{(1-q^{b+1})(1-q^{b+2})\cdots (1-q^{b+a-1})}{[a-1]_{q^a}}[a]_{q^2}[a]_{q^3}\cdots[a]_{q^{a-1}}
\end{align*}

Observe that the fraction is similar to the $q^a$-count of the lattice
points inside a simplex of size $b/a$, and that the product of $[a]_{q^i}$ is a $q$ analog of $a^{a-2}$.

\subsubsection{}
This algebraic expression is suggestive of the simplex of
$(a,b)$-cores.  The lattice of $a$-cores has index $a$ within
the standard lattice.  The sublattice $\Lambda_T=(a\Z)^{a-1}$, has index $a^{a-1}$ inside the standard lattice, and hence $a^{a-2}$ within the lattice of $a$-cores.

The intersection of each coset $\mathfrak{c}$ of $\Lambda_T$ with the simplex of
$(a,b)$-cores is a $k\Delta_{a-1}$, where $k$ is slightly smaller than $b/a$, and depends on $b$ and $\mathfrak{c}$.

It appears that
$\Cat_{a,b}(q)$ is $q^a$ counting the lattice points in each coset $\mathfrak{c}$,
but then shifting the result by a factor of $q^{\iota(\mathfrak{c})}$ for some $\iota(\mathfrak{c})$.

Algebraically, this suggests
\begin{conjecture} \label{conj:cosets}
There is an \emph{age} function $\iota$ on the cosets $\mathfrak{c}\in\Lambda/\Lambda_T$, so that
$$\sum_{\mathfrak{c}\in\Lambda/\Lambda_T} q^{\iota(\mathfrak{c})}=[a]_{q^2}[a]_{q^3}\cdots[a]_{q^{a-1}}$$
 and
$$\Cat_{a,b}(q)=\sum_{\mathfrak{c}\in\Lambda/\Lambda_T} q^{\iota(\mathfrak{c})} {b/a-s(\mathfrak{c},b)+a-1\brack a-1}_{q^a}$$
where the $q^a$ binomial coefficient $q^a$-counts the points in $\mathfrak{c}\cap \mathcal{SC}_{a-1}(b)$.
\end{conjecture}

\begin{remark}
We could not find an obvious candidate for an explicit form of $\iota$ in general.


\end{remark}

\begin{remark}
Conjecture \ref{conj:cosets} was motivated in part by Chen-Ruan cohomology \cite{CR, ALR}, which has found applications to the Ehrhart theory of rational polytopes \cite{Stapledon}.  Chen-Ruan cohomology $H^*_{CR}(\mathcal{X})$ is a cohomology theory for an orbifold (or Deligne-Mumford stack) $\mathcal{X}$.   As a vector space, $H^*_{CR}(\mathcal{X})$ is the usual cohomology of a disconnected space $\mathcal{IX}$.  One component $C_0$ of $\mathcal{IX}$ is isomorphic to $\mathcal{X}$.  The other components $C_\alpha, \alpha\neq 0$ are called \emph{twisted sectors} and are (covers of) fixed point loci in $\mathcal{X}$.  The pertinent feature for us is that the grading of the cohomology of the twisted sectors are \emph{shifted} by rational numbers, $\iota(\alpha)$, that is
$$H^k_{CR}(\mathcal{X})=\bigoplus_{\alpha} H^{k-\iota(\alpha)}(C_\alpha)$$
The function $\iota$ is known as the ``degree shifting number'' or ``age''.

Orbifolds could potentially be connected to our story through toric geometry, and the well known correspondence between lattice polytopes and polarized toric varieties.  When the polytope is only rational, in general the toric variety is an orbifold. The simplex of $(a,b)$-cores in $\Lambda_a$ corresponds  orbifold $[\proj^a/\Z_a]$.  More specifically, there is an torus equivariant orbifold line bundle $\mathcal{L}$ over $\proj^a/\Z_a$, so that the lattice points in $\SC(a,b)$ correspond to the torus equivariant sections of $\mathcal{L}^b$.

In the fan point of view, the cosets of the lattice correspond exactly to group elements of isotropy groups, and hence to twisted sectors.  

This discussion is rather vague, and at this point, there is no concrete connection between $\Cat_{a,b}(q)$ and the geometry of the orbifold $\proj^a/\Z_a$ it would be very interesting to find one.

\end{remark}

Note that if Conjecture \ref{conj:cosets} holds, it would give
another proof, presumably more combinatorial, that $\Cat_{a,b}(q)$ are all positive.  Furthermore, Conjecture \ref{conj:cosets} suggests:

\begin{conjecture}
For every residue class $r, 0\leq r< a$, the coefficients of $q^{ak+r}$ in $\Cat_{a,b}(q)$ are unimodal.
\end{conjecture}

\subsubsection{Examples}

\begin{example}
By expanding both sides, it is straightforward to check the identities
\begin{align*}
\Cat_{3,3k+1}(q) & ={k+2\brack 2}_{q^3}
+q^2{k+1 \brack 2}_{q^3}
+q^4{k+1\brack 2}_{q^3} \\
\Cat_{3,3k+2}(q)&=
{k+2\brack 2}_{q^3}
+q^2{k+2\brack 2}_{q^3}
+q^4{k+1\brack 2}_{q^3}
\end{align*}
\end{example}

\begin{example}
When $a=4$ and $b=4k+1$,

\begin{align*}
\Cat_{4, 4k+1}(q)=&
{k+3\brack 3}_{q^4}&+&q^4{k+2\brack 3}_{q^4}&+&q^8{k+2\brack 3}_{q^4}&+&q^{12}{k+1\brack 3}_{q^4} \\
+&q^5{k+2\brack 3}_{q^4}&+&q^9{k+2\brack 3}_{q^4}&+&q^9{k+1\brack 3}_{q^4}&+&q^{13}{k+1\brack 3}_{q^4} \\
+&q^2{k+2\brack 3}_{q^4}&+&q^6{k+2\brack 3}_{q^4}&+&q^6{k+1\brack 3}_{q^4}&+&q^{10}{k+1\brack 3}_{q^4}\\
+&q^3{k+2\brack 3}_{q^4}&+&q^7{k+2\brack 3}_{q^4}&+&q^{11}{k+1\brack 3}_{q^4}&+&q^{15}{k+1\brack 3}_{q^4}
\end{align*}
Here the terms have been grouped so that the coefficients on each line have the same residue mod 4, making it easy to verify the unimodality conjecture.

\end{example}

\section{Toward \texorpdfstring{$(q,t)$}{(q,t)}-analogs} \label{sec:qt}

We now turn toward applying the lattice-point viewpoint toward the $(q,t)$-analog $\Cat_{a,b}(q,t)$, original defined in terms of simultaneous cores in \cite{AHJ}.

\subsection{Results}
Our main result is that the statistics $\ell$ and $\sk$ in the definition of $\Cat_{a,n}(q,t)$ are piecewise linear functions on the simplex of cores $\SC_a(b)$.  

\begin{proposition}
\label{prop:length}
$$\ell(\mathbf{x})=-\frac{a-1}{2}+a\max x_i$$
\end{proposition}

\begin{proposition}
\label{prop:skewlength}
Let $\lfloor x\rfloor_0=\max\left(0, \lfloor x\rfloor\right)$.
Then
$$\sk(\mathbf{x})=\sum_{i,j=0}^a \lfloor x_i-x_j\rfloor_0 - \lfloor x_i-x_j-b/a\rfloor_0$$
\end{proposition}

\subsubsection{Chambers of linearity}

The chambers of linearity of $\ell$ are unions of chambers of the \emph{braid hyperplane arrangement}; the chambers of linearity for $\sk$ are chambers of a deformation of the braid arrangement called the  \emph{Catalan hyperplane arrangement}.  See \cite{StanleyHyperplane} for an introduction from a combinatorial point of view.

\begin{definition}
The $A_{a-1}$ hyperplane arrangement is the set of the $\binom{a}{2}$ hyperplanes $x_i=x_j$ in the $a-1$ dimensional vector space $\sum x_i=0$.

There are $a!$ regions of the $A_{a-i}$ arrangement, which are indexed by permutations $\sigma$; the region indexed by $\sigma$ is where $x_{\sigma(0)}<x_{\sigma(1)}<\cdots<x_{\sigma(a)}$.
\end{definition}

\begin{definition}
A hyperplane arrangement $\mathcal{A}^\prime$ is a \emph{deformation} of an arrangement $\mathcal{A}$ if every hyperplane in $\mathcal{A}^\prime$ is parallel to one in $\mathcal{A}$.
\end{definition}

\begin{definition} \label{def:arrangements}
The \emph{Catalan arrangement} $\mathcal{C}_a$ is the union of the $3\binom{a}{2}$ hyperplanes $x_i-x_j\in\{-1,0,1\}, i<j$.
\end{definition}

The name \emph{Catalan arrangement} comes from the fact that $\mathcal{C}_a$ has $a!C_a$ regions.

We have already seen the hyperplanes in the Catalan arrangement appearing.  If $b\mathcal{C}_a$ denotes the Catalan arrangement scaled by $b$ (so $x_i-x_j\in\{-b,0,b\}$), then the hyperplanes that define the simplex $\SC_a(b)$ of $(a,b)$-cores are in $b\mathcal{C}_a$.

From Proposition \ref{prop:length}, it is clear that length is linear on each chamber of the braid arrangement.  

The formula for $\sk$ given in Proposition \ref{prop:skewlength} is not piecewise linear on the vector space $V_a$.  However, when we restrict to the lattice $\Lambda+s$, the $x_i$ only change by an integers, and so on this restricted domain $\sk$ is indeed piecewise linear.  There is a piecewise linear function on all of $V_a$ that agrees with our $\sk$ on the points of $\Lambda+s$, but it is more complicated to write down.  In particular, it is not $S_a$ invariant, while our formula for $\sk$ is.

\subsubsection{Examples: largest and smallest $(a,b)$-cores}
As a basic check, we now illustrate that Propositions \ref{prop:length} and \ref{prop:skewlength} give the correct results for the smallest and large $(a,b)$-cores; we will use these results later.

\begin{example}[The empty partition]
The empty partition corresponds to the vector $s$; recall $s_i=i/a-(a-1)/(2a)$.  
  The largest of the $s_i$ is $s_{a-1}=(a-1)/(2a)$, and so $\ell(s)=a(a-1)/(2a)-(a-1)/2=0$.  

Since $s_i-s_{i-1}=1/a$, we have $s_i-s_j<1$, and so $\lfloor s_i-s_j\rfloor_0=0$,  verifying that $\sk(s)=0$.

\end{example}

\begin{example}[The largest $(a-b)$-core]

The largest $(a,b)$-core $\lambda^M$ is the one vertex of $\SC_a(b)$ that is a lattice point.  Its coordinates are some permutation of
$bs=(bs_0,bs_1,\cdots bs_{a-1})$. Since $\sk$ is $S_a$ invariant we may assume it is $bs$.  

It is immediate that:
$$\ell(\lambda^M)=-\frac{a-1}{2}+ab\frac{a-1}{2a}=\frac{(a-1)(b-1)}{2}$$

Verifying $\sk(\lambda^M)=(a-1)(b-1)/2$ is more complicated.  We have 
$$\sk(\lambda^M)=\sum_{i<j} \left\lfloor \frac{bj}{a}-\frac{bi}{a}\right\rfloor
-\left\lfloor \frac{bj}{a}-\frac{bi}{a}-\frac{b}{a}\right\rfloor$$
The summand depends only on the difference $k=j-i$, and is equal to $\lfloor kb/a\rfloor-\lfloor(k-1)b/a\rfloor$.

There are $(a-1)$ pairs $(i,j)$ with $i-j=1$, and in general $a-k$ pairs with $i-j=k$, and so we have
\begin{align*}
\sk(\lambda^M) &=\sum_{k=1}^{a-1}(a-k)\left\lfloor \frac{b}{a} k\right\rfloor-(a-k)\left\lfloor\frac{b}{a}(k-1)\right\rfloor \\
&=\sum_{k=1}^{a-1} \left\lfloor \frac{b}{a} k\right\rfloor \\
&=\sum_{k=1}^{a-1} \frac{b}{a}k-\sum_{k=1}^{a-1}\left\langle \frac{b}{a}k\right\rangle \\
&=\frac{b}{a} \frac{(a-1)a}{2}-\frac{1}{a}\frac{(a-1)a}{2} \\
&=\frac{(a-1)(b-2)}{2}
\end{align*}
where the second line follows from reindexing the second sum, the third line applies $\lfloor x\rfloor=x-\langle x\rangle$, and the fourth line applies $\sum i=n(n+1)/2$ and the fact that, since $a$ and $b$ are relatively prime, $kb$ takes on every residue class mod $a$ exactly once as $k$ ranges from 1 to $a$.

\end{example}

\subsection{Length and Skew Length are piecewise linear}
In this section we prove Propositions \ref{prop:length} and \ref{prop:skewlength}.

\subsubsection{Proof of Proposition \ref{prop:length} - length is piecewise linear}

\begin{proof}
We first translate $\ell(\lambda^S)$ into fermionic language.
Let $e$ be the lowest energy state of $S$ that is not occupied by an electron.  Then $\ell(\lambda^S)$ is the number of electrons with energy greater than $e$.

Recall that the highest energy occupied state on the $i$th runner is $-ac_i-i-1/2$, and so the lowest unoccupied state is $a$ higher, and hence $e=\min_i -ac_i-i-1/2+a$.  

Let $m$ be the runner of the $a$-abacus that has the lowest unoccupied energy state.  For $i\neq m$, there are roughly $c_m-c_i\geq 0 $ electrons on the $i$th runner that have energy great than $e$.  The exact number depends  on which of $i$ and $m$ is bigger: if $i<m$, there are exactly $c_m-c_i$ such electrons, while if $i>m$, there are only $c_m-c_i-1$ such electrons. 

 There are $a-1-m$ runners with $i>m$, and hence we have
\begin{align*}
\ell(\core_a(\mathbf{c})) & =-(a-1-m)+\sum_{i\neq m} c_m-c_i \\
 &=-(a-1-m)+ac_m
 \end{align*}
 where the second line follows by adding $\sum c_i=0$ to the expression.

Since $x_i=c_i+i/a-(a-1)/(2a)$, it follows that
$$
\ell(\core_a(x))= -(a-1)/2+a\max x_i 
$$

\end{proof}

\subsubsection{Proof of Proposition \ref{prop:skewlength} - skew length is piecewise linear}

\begin{definition} \label{def:skewij}
For $\lambda$ and $(a,b)$-core, let $\sk^T_{i,j}(\lambda)$ be the number of cells in the $i$th $a$-part with unoccupied state on the $j$th runner.

Furthermore let $\sk_{ij}^S(\lambda)$ be the number of such cells with hook length less than $b$, and $\sk_{ij}^B(\lambda)$ be the number of such cells with hook length greater than $b$.
 
Here, $T, S$ and $B$ are short for \emph{total, small} and \emph{big}.
\end{definition}

From Definition \ref{def:skewij} it is clear that
$$\sk(\lambda)=\sum_{i\neq j}\sk^S_{ij}(\lambda)$$
$$\sk_{ij}^S(\lambda)=\sk_{ij}^T(\lambda)-\sk^B_{ij}(\lambda)$$
and so Proposition \ref{prop:skewlength} follows from the following lemma.

\begin{lemma}
Let $\lambda=\core_a(\mathbf{x})$ be an $(a,b)$-core.  Then:
\begin{align*}
\sk_{ij}^T(\lambda)&=\lfloor x_i-x_j\rfloor_0 \\
\sk_{ij}^B(\lambda)&=\lfloor x_i-x_j-b/a\rfloor_0
\end{align*}
\end{lemma}

\begin{proof}
Recalling that cells are in bijection with pairs $(e, f)$, with $e>f$ energy levels, $e$ filled and $f$ empty, we see that $\sk^T_{ij}$ counts pairs $(e, f)$ with $e$ the highest energy level on the $i$th runner, $f<e$ any empty state on the $j$th runner.  In other words, $\sk^T_{ij}(\lambda)$ is the number of unoccupied states on the $j$th runner with energy less than $e$.

Recalling that the highest energy electron on the $i$th runner has energy $e_i=-ac_i-i-1/2$, and that the energy of each state to the left increases by $a$, we have  
\begin{align*}
\sk_{ij}^T(\lambda)&=q_a\left(-ac_i-i-1/2-(-ac_j-j-1/2)\right) \\
&=q_a(-a(x_i-x_j)) \\
&=\lfloor x_j-x_i\rfloor_0
\end{align*}

For $\sk_{ij}^B(\lambda)$, we want hooklengths of size at least $b$, so begin by reducing the energy of the first electron on the $i$th runner by $b$.  We now want to count ways of moving the resulting electron onto the $j$th runner, and so by our calculation of $\sk_{ij}^T(\lambda)$ we immediately have
$$\sk_{ij}^B(\lambda)=\lfloor x_j-x_i-b/a\rfloor_0$$
\end{proof}

\subsection{Symmetry}

Before applying piecewise linearity to the Symmetry and Specialization conjectures, we exploit the fact that $\ell$ and $\sk$ are $S_a$ symmetric in the $\mathbf{x}$-coordinates.

\subsubsection{The Dominant Cone $\dominant$}

Let $\dominant$ denote the dominant chamber $$\dominant=\{x\in V_a|x_0< x_1< x_2<\cdots < x_{a-1}\}$$
Then $\dominant$ is a fundamental domain for the action of $S_a$ on $V_a$.  We use $\polyh$ to denote the polyhedron $\dominant\cap\SC_{a}(b)$.  

Two vertices of $\polyh$ are particularly important to us: $0$ and  $x_\infty=b\cdot s$.

Consider the quotient map from $\SC_a(b)$ to $\polyh$.  A generic point near the origin in $\polyh$ has $a!$ preimages in $\SC_a(b)$.  However, as we cross the walls of the Catalan arrangement the number of preimages drops -- a point near $x_\infty$ has only $a$ preimages -- one near each vertex of $\SC_a(b)$. See Figure \ref{fig-quotient}.

\subsubsection{A refined lattice}

Now consider the image of $\Lambda_S\cap \SC_{a}(b)$ under the $S_a$ action as a subset of $\polyh$.   Since a point $x=(x_1,\dots,x_a)\in\Lambda+s$ must have have distinct coordinates, each point of $\Lambda_S\cap \SC_a(b)$ has a unique representative in $\polyh$, even though the quotient map is not injective.

Consider the action of $\Z_a\subset S_a$ that cylicly permutes the coordinates.  We have seen that the image of $\Lambda+s$ under the action of this $\Z_a$ action is a lattice $\Lambda_R$, and that $\SC_a(b)$ is integral with respect to $\Lambda_R$.  In fact, the braid arrangement, and hence $\polyh$, are unimodular with respect to $\Lambda_R$.  

\begin{definition} The \emph{rotated lattice} $\Lambda_R$ is the $a-1$ dimensional lattice
$$\Lambda_R=\{\mathbf{z}=(z_1,\dots z_a)|\mathbf{z}\in\Z^a+\mathbb{Z}(1/a,1/a,\dots, 1/a), \sum_{i=1}^a z_i=0\}$$
\end{definition}

\begin{definition}
For $1\leq i\leq a-1$, let $$v_i=\left(\underbrace{\frac{i}{a}-1,\dots \frac{i}{a}-1}_{i \text{ times}},\underbrace{ \frac{i}{a},\dotsm\frac{i}{a}}_{a-i \text{ times}}\right)$$
\end{definition}

One can see that the $v_i$ generate the lattice $\Lambda_R$ and that each $v_i$ spans one of the rays of $\dominant$ at $0$, and so the braid arrangement is unimodar with respect to $\Lambda_R$.

This means that locally near $x_0, \polyh=x_0+\sum t_i v_i$, with $t_i\in\Z, \_i\geq 0$, while near $x_\infty$ we have $\cone=x_\infty-\sum t_iv_i$ .

Because $\ell$ is a linear function on $\dominant$ it is immediate from the definitions of $\ell$ and $v_i$ that, for any point $x\in\cone$ we have  
$$\ell(x+v_i)=\ell(x)+i$$

Because the difference of two entries of $v_i$ is $0$ or $1$, we see that $\sk$ is a piecewise linear function when restricted to the elements of any translate of $\Lambda_R$.

The dependence of $\sk$ on $v_i$ depends on which chamber of the Catalan arrangement we are in.  Near $x_0$, we have $$\sk(x)=\sum_{i<j} x_j-x_i$$ and so
$$\sk(x+v_i) = \sk(x)+i(a-i)$$
However, near $x_\infty$, we have that $x_j-x_i>b$ if $j\neq i+1$, so 
$$\sk(x)=\sum_i x_{i+1}-x_i=x_a-x_1$$ and
$$\sk(x+v_i)=\sk(x)+1$$

This discussion is summarized as follows:

\begin{lemma} \label{lem:table}
Let $f\in\{\ell,\sk\}$.  For $x$ near $x_0$, let $\Delta_if=f(x+v_i)-f(x)$, and near $x_\infty$ let $\Delta_if=f(x-v_i)-f(x)$.  Then:

$$\begin{array}{r|cc}
 & \Delta_i\ell & \Delta_i\sk^\prime \\
\hline
0 & i & -i(a-i) \\
\infty & -i & 1
\end{array}
$$
\end{lemma}

\subsubsection{Orbifold cosets}

The quotient of $\Lambda_S$ by the cyclic $\Z_a$ action results in the refined lattice $\Lambda_R$, but we want to quotient out by the full $S_a$ action.  The resultant set of points is not itself a lattice, but consists of cosets of the $\Lambda_R$ lattice, which we call \emph{orbifold cosets}.  

We will use $\Lambda_{\mathcal{O}}$ denote the set $S_a(\Lambda_S)$; it consists of $(a-1)!$ cosets of $\Lambda_R$.  We denote the set of these cosets by $\mathcal{OC}$.

As the number of preimages of a point in $\polyh$ depends on the chamber of the Catalan arrangement, the number of orbifold cosets does as well. Near $x_{\infty}$, there is only be one orbifold coset, while near $0$ there are $(a-1)!$, and the chambers in between vary between these two extremes.   

\begin{figure}
\caption{Simplex of cores with chambers of linearity, and the quotient by $S_3$}
\label{fig-quotient}
\includegraphics{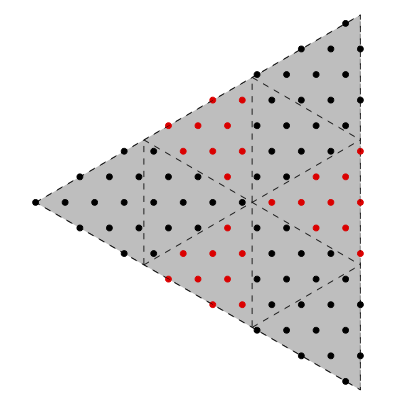}
\includegraphics{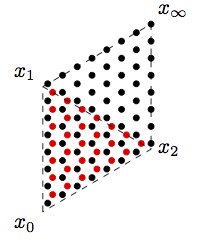}
\end{figure}

\subsection{Rationality and \texorpdfstring{$\Cat_{3,b}(q,t)$}{Cat{3,b}(q,t)}}

It is an immediate corollary of the piecewise linearity of the length and the skew length that, for fixed $a$, and for $b=ak+\delta$ in a fixed residue class $\delta$ mod $a$, we have that $\Cat_{a,b}(q,t)$ is a rational function of $q$ and $t$, and the denominator can be written so the exponents depend linearly on $b$ and $\delta$.

We explicitly compute these rational functions when $a=3$:

\begin{proposition} \label{prop:qt3}
Let $b=3k+1+\delta$, where $\delta\in\{0,1\}$.  Then:

$$\Cat_{3,b}(q,t)=\frac{t^{3k+\delta}}{(1-qt^{-1})(1-qt^{-2})}+\frac{q^kt^k(q+t+q^\delta t^\delta)}{(1-q^{-1}t^2)(1-q^2t^{-1})}+\frac{q^{3k+\delta}}{(1-tq^{-1})(1-tq^{-2})}$$

\end{proposition}

\begin{proof}
After the $S_3$ action, there are two chambers of linearity for $\ell$ and $\sk$, which we call Chamber I and Chamber II.  Both are triangles; Chamber I has vertices $(0,0,0), x_1=(-2b/9, b/9, b/9)$ and $x_2=(-b/9, -b/9, 2b/9)$.  Chamber II shares vertices $x_1$ and $x_2$ with Chamber I, and has third vertex $x_\infty=(-b/3,0,b/3)$.  Chamber I has two orbifold cosets, while Chamber II has only one.  See Figure \ref{fig-quotient}.  

Thus, we can express $\Cat_{3,b}(q,t)$ as the sum of three indicator functions of rational polytopes.  We write this indicator functions as a sum of contributions from the vertices using the Brion decomposition (rather than the Lawrence-Varchenko decomposition).  The Brion decomposition says that the the indicator function of a rational polyhedron is the \emph{positive} sum of the \emph{inward} pointing indicators of the cone at each vertex:

\begin{center}
  \begin{tikzpicture}
  \draw[fill=blue!20] (0,0) node [below left] {$A$} --(1,0) node [below right] {$B$}--(2,1) node [above right] {$C$}--(0,1) node [above left] {$D$}--cycle;

\draw (2.5,.5) node {$=$};

\begin{scope}[xshift=3cm]
    \draw (0,0) node [below left] {$A$};
    \draw[-triangle 45] (0,0)-- (0,1);
    \draw[-triangle 45] (0,0)-- (1, 0);
\end{scope}

\begin{scope}[xshift=5cm]
    \draw (-.6,.5) node {$+$};
    \draw (1,0) node [below right] {$B$};
    \draw[-triangle 45] (1,0)-- (0,0);
    \draw[-triangle 45] (1,0)-- (2, 1);
\end{scope}

\begin{scope}[xshift=8cm]
    \draw (-.6,.5) node {$+$};
    \draw (1,1) node [above right] {$C$};
    \draw[-triangle 45] (1,1)-- (0,0);
    \draw[-triangle 45] (1,1)-- (0, 1);
\end{scope}

\begin{scope}[xshift=10cm]
    \draw (-.6,.5) node {$+$};
    \draw (0,1) node [above left] {$D$};
    \draw[-triangle 45] (0,1)-- (0,0);
    \draw[-triangle 45] (0,1)-- (1, 1);
\end{scope}

  \end{tikzpicture}

\end{center}

In what follows, we determine the contribution of each of the four vertices $x_0,x_1,x_2,x_\infty$ to the Brion decomposition of $\Cat_{3,b}(q,t)$.

At $x_0$, the rays of the cone are $v_1$ which has $(\ell, \sk^\prime)$ weight $(1, -2)$, and $v_2$, which has $(\ell, \sk^\prime)$ weight $(2,-2)$.  Thus, the denominator at $0$ is $1/(1-qt^{-2})1/(1-q^2t^{-2})$.

The closest point to $0$ in the trivial orbifold coset is $x_0=(-1/3,0,1/3)$, which has weight $t^{3k+\delta}$, and the vertex closest to the origin in the nontrivial orbifold coset is $x_0^\prime=(-2/3,0,2/3)$, which has weight $qt^{3k+\delta-1}$.

Thus, the total contribution at $x_0$ is
$$\frac{t^{3k+\delta}+qt^{3k+\delta-1}}{(1-qt^{-2})(1-q^2t^{-2})}=\frac{t^{3k+\delta}(1+qt^{-1})}{(1-qt^{-2})(1-q^2t^{-2})}
=\frac{t^{3k+\delta}}{(1-qt^{-2})(1-qt^{-1})}$$
the first term in Proposition \ref{prop:qt3}.

At $x_1=(-2b/9, b/9, b/9)$, the rays pointing inward to Chamber I are $-v_1$and $v_2-v_1$; reading from Lemma \ref{lem:table}, we see that these rays have $(\ell, \sk^\prime)$ weights $(-1, 2)$ and $(1, 0)$, respectively.  Thus the denominator from the Chamber I cosets is $(1-q^{-1}t^2)(1-q)$.  The rays pointing inward to Chamber II at $x_1$ are $v_2-v_1$ and $v_2$; hence the denominator of the Chamber II contribution is $(1-q^2t^{-1})(1-q)$.

To find the numerators, we find the closest points in each chamber and relevant orbifold coset to $(-2b/9, b/9, b/9)$.  This information is summarized in the following table, which lists four points, the value of $\ell$ and $\sk$ on each point, the coset it belongs to, and which chamber it contributes to when $\delta=0$ and when $\delta=1$.  So, for instance, point $y$ is in Chamber II when $\delta=0$, but crosses to Chamber I when $\delta=1$, and point $z$ is the closest point to $x_1$ in Chamber II when $\delta=1$, but doesn't contribute when $\delta=0$.

$$1111
\begin{array}{rrrrr|rr}
\text{point} & \shortstack{ \text{Shift from}\\ (-2k/3,k/3,k/3)} & \ell & \sk & \text{Coset} & \delta=0  &\delta=1  \\
\hline
w  & (1/3, -1/3, 0) & k-1 & 2k-2 & 1 & I & \\
x  & (0, -1/3, 1/3) & k &2k-1 & 2 & I & I \\ 
y & (-1/3, 0, 1/3) & k &2k & 1 & II & I \\
z & (-2/3, 0, 2/3) & k+1 & 2k+1 & 2 &   & II 
\end{array}
$$

From the table and the description of the rays of Chamber I and II at this point, we see that the contribution of $x_1$ to the Brion decomposition of $\Cat_{3,3k+1+\delta}(q,t)$ is 
$$\frac{q^{k-1+\delta}t^{k+2}+q^kt^{k+1}}{(1-q^{-1}t^2)(1-q)}+\frac{q^{k+\delta}t^k}{(1-q^2t^{-1})(1-q)}$$
which algebraic manipulation shows is equal to
$$\frac{q^kt^k(q+t+q^\delta t^\delta)}{(1-q^{-1}t^2)(1-q^2t^{-1})}$$
the middle term in Proposition \ref{prop:qt3}.

At $x_2=(-b/9, -b/9, 2b/9)$, the rays pointing in to Chamber I are $-v_2$ and $v_1-v_2$, which have $(q,t)$ weight $(-2, 2)$ and $(-1, 0)$, respectively here.  The rays pointing into Chamber II at $x_2$ are $v_1$ and $v_1-v_2$, which have $(q,t)$ weight $(1,-1)$ and $(-1,0)$ here.

$$
\begin{array}{rrrrr|rr}
\text{point} & \shortstack{ \text{Shift from}\\ (-k/3,-k/3,2k/3)} & \ell & \sk & \text{Coset} & \delta=0  &\delta=1  \\
\hline
w  & (0, 1/3, -1/3) & 2k-2 & 2k-2 & 1 & I & \\
x  & (-1/3, 1/3, 0) & 2k-1 &2k-1 & 2 & I & I \\ 
y & (-1/3, 0, 1/3) & 2k &2k & 1 & II & I \\
z & (-2/3, 0, 2/3) & 2k+1 & 2k+1 & 2 &   & II 
\end{array}
$$

From the table and the description of the rays of Chamber I and II at $x_2$, we see that the contribution of $x_2$ to the Brion decomposition $\Cat_{3, 3k+1+\delta}(q,t)$ is:
$$\frac{q^{2k-2+\delta}t^{k+2}+q^{2k-1+\delta}t^{k+1}}{(1-q^{-2}t^2)(1-q^{-1})}+\frac{q^{2k+\delta}t^k}{(1-qt^{-1})(1-q^{-1})}$$
which algebraic manipulation shows vanishes.

At $x_\infty$, there is only one orbifold coset, and the inward pointinting vectors are $-v_1$ and $-v_2$ which have $(q,t)$ weight $(-1,1)$ and $(-2,1)$ here.  Thus, the contribution of $x_\infty$ is 
$$\frac{q^{3k+\delta}}{(1-tq^{-1})(1-tq^{-2})}$$
the last term in Proposition \ref{prop:qt3}.

\end{proof}

That the contributions of some of the vertices vanish, leaving just four terms, and the specialization conjecture, suggest that perhaps that is a different parametrization of simultaneous cores as lattice points inside a rational simplex so that $\ell$ and $\sk$ become linear functions on the simplex.

\subsection{Low degree \texorpdfstring{$(q,t)$}{(q,t)}-symmetry}
In this section we show, for all $(a,b)$, that $(q,t)$-symmetry holds when the degree of one of the monomials are small.

More precisely, we show
\begin{corollary}[Low degree $(q,t)$-symmetry]
$$[t^k]\Cat_{a,b}(q,t)=[t^k]\Cat_{a,b}(t,q)$$
for $k$ sufficiently small (compared to $a$ and $b$).
\end{corollary}

\subsubsection{Contribution near $x_\infty$}
As we saw for $a=3$, near $x_\infty$ there is only one orbifold coset of the rotated lattice $\Lambda_R$.  We have seen that $\ell(x_\infty)=(a-1)(b-1)/2$, $\sk(x_\infty)=0$.  Reading off how adding multiples of $v_i$ changes $\ell$ and $\sk^\prime$ from Lemma \ref{lem:table}, we see that the low $t$-degree terms of $\Cat_{a,b}(q,t)$ are:

$$q^{(a-1)(b-1)/2}\prod_{k=1}^{a-1}\frac{1}{(1-tq^{-k})}$$

\subsubsection{Contribution near 0}
Let $\cone_0=\{x_0+\Lambda_R\}\cap\cone$ be the set of points in the intersection of $\cone$ and the orbifold coset of $\Lambda_R$ containing $x_0$.

From Lemma \ref{lem:table} and the values of $\ell, \sk^\prime$ on $x_0$, we have

$$\sum_{p\in\cone_0} q^{\sk^\prime(p)}t^{\ell(p)}=q^{(a-1)(b-1)/2}\prod_{k=1}^{a-1} \frac{1}{(1-t^kq^{-k(a-k)})}$$

To figure out the entire contribution to $\Cat_{a,b}(t,q)$, we must figure out the contribution from the other $(a-1)!$ orbifold cosets of $\Lambda_R$ near $x_0$.  

Since $\cone$ is integral at $0$ with respect to $\Lambda_R$, each orbifold coset $\gamma$ of $\Lambda_R$  had a unique minimal representative $x_\gamma$, so that the points in $\gamma\cap\cone$ str $x_\gamma+\left(\Lambda_R\cap \cone\right)$,  and the contribution near $0$ of the points in $\gamma$ is 
$$q^{\sk^\prime(x_\gamma)}t^{\ell(x_\gamma)}\prod_{k=1}^{a-1}\frac{1}{1-t^kq^{-k(a-k)}}$$

Thus, low degree symmetry follows from
\begin{proposition} \label{prop:orbcosets}
$$\sum_{\gamma\in\OC} q^{\sk^\prime(x_\gamma)}t^{\ell(x_\gamma)}=\prod_{k=1}^{a-1}[k]_{q^{-(a-k)}t}$$ 

\end{proposition}
  
\subsubsection{Proof of Proposition \ref{prop:orbcosets}}

We break the proof of Proposition \ref{prop:orbcosets} into two lemmas.  The first establishes a bijection between the orbifold cosets $\gamma$ and permutations in $S_{a-1}$, and identifies permutation statistics that correspond to $\ell$ and $\sk$ under this bijection.  The second lemma shows that these permutation statistics have the proper distribution.  Before stating these lemmas, we introduce these permutation statistics.

\subsubsection{Permutation Statitistics}
The permutation statistics we need are defined in terms of descents and inversions.
\begin{definition} \label{def:standardpermutations}
For $\sigma\in S_n$, let 
$$\DES(\sigma)=\left\{i\in [1, n-1] \Big | \sigma(i)>\sigma(i+1)\right\}$$
We use $\des(\sigma)$ to denote $|\DES(\sigma)|$, and 
$$\maj(\sigma)=\sum_{i\in\INV(\sigma)} i$$

Recall that 
$$\inv(\sigma)=\left|\left\{(i,j)\big| 1\leq i<j\leq n, \sigma(i)>\sigma(j)\right\}\right|$$
\end{definition}

The lesser known statistic we need is the $\emph{size}$ of $\sigma$, written $\siz(\sigma)$: 

\begin{definition} \label{def:sizepermutation}
$$\siz(\sigma)=\left(\sum_{i\in\DES(\sigma)} (n+1-i)i\right)-\inv(\sigma)$$
\end{definition}

Our motivation for the definition of $\siz$ are the following two lemmas, which together immediately prove Proposition \ref{prop:orbcosets}

\begin{lemma} \label{lem:sizmaj1}
There is a labeling of the orbifold cosets by permutations $\sigma\in S_{a-1}$, so that if $v_\sigma$ be the minimum vector in the coset labeled by $\sigma$, then:
\begin{align*}
\ell(v_\sigma)&=\maj(\sigma) \\
\sk(v_\sigma)&=\siz(\sigma)\\
\end{align*}

\end{lemma}

\begin{lemma} \label{lem:sizmaj2}
$$\sum_{\sigma\in S_n} q^{\siz(\sigma)} t^{\maj(\sigma)}=\prod_{k=1}^n [k]_{q^{n+1-k}t}$$
\end{lemma}

\begin{remark}
The name \emph{size} was chosen in reference to the size of a partition: by Lemma \ref{lem:sizmaj2}, for fixed $k$ and $\ell$, as $n$ grows large the number of permutations $\sigma\in S_n$ with $\maj(\sigma)=\ell$ and $\siz(\sigma)=k$ stabilizes to the number of partitions with length $\ell$ and size $k$.
\end{remark}

Note that $\siz$ is quadratic in the descent positions.  Such statistics have been considered by Bright and Savage in \cite{BS}.  In particular, they introduce the statistic 
$$\sqin(\sigma)=\inv(\sigma)+\sum_{i\in\DES(\sigma)} i^2$$
and prove in Theorem 4.4 that
$$\sum_{\sigma\in S_n} q^\sqin(\sigma)t^{\maj(\sigma)}=\prod_{k=1}^n [k]_{tq^k}$$
Substituting $q^{-1}$ for $q$ and $tq^{n+1}$ for $t$ in their Theorem 4.4 gives exactly our Lemma \ref{lem:sizmaj2}.

\subsection{Proof of Lemma \ref{lem:sizmaj1}}
\subsubsection{Bijection between $S_{a-1}$ and orbifold cosets}
First, we determine a bijection between orbifold cosets and $S_{a-1}$.

Let $w\in\Lambda\cap \cone$, and define $\sigma^w$ by 
$$\frac{\sigma^w_i}{a}=\langle w_{i}-w_a\rangle$$
As $w\in S_a\Lambda_C$, we see $\sigma^w$ is a permutation in $S_{a-1}$.

Since the entires of the $v_i$ all have the same entries modulo $1$, we see that $\sigma^{w+v_i}=\sigma^w$; that is, $\sigma^w$ is constant on the orbifold cosets.  

It is not hard to see that this map is surjective, and hence a bijection between orbifold cosets and $S_{a-1}$.

\subsubsection{Smallest vector in each coset}
We now describe the minimal element $x^\sigma$ in the orbifold coset corresponding to $\sigma$.  

Being the minimal vector $x^\sigma$ in a coset means that $x^\sigma-v_i\notin\dominant$ for all $i$, which is equivalent to
$$x^\sigma_i+1>x^\sigma_{i+1},\quad 1\leq i\leq a-1$$

To find $x^\sigma$ we first define a vector $w^\sigma$ satisfying
$$w^\sigma_i<w^\sigma_{i+1}<w^\sigma_i+1$$
$$\left\langle w_i-w_a\right\rangle=\frac{\sigma_i}{a}$$
but does not satisfy $\sum w_i^\sigma=0$, we then subtract the approproiate multiple of $(1/a,\dots, 1/a)$ to get $v^\sigma$.

We need $w^\sigma_{i+1}>w^\sigma_i$ and $\langle w^\sigma_{i+1}-w^\sigma_i\rangle=\langle\sigma_{i+1}/a-\sigma_i/a\rangle$, and so we set
$$w^\sigma_{i+1}=w^\sigma_i+\frac{\sigma_{i+1}-\sigma_i}{a}+\des_i(\sigma)$$
where we have conventionally set $w^\sigma_0=\sigma_0=0, \sigma_a=a$.

Then
$$x_i^\sigma=w_i^\sigma-\frac{1}{a}\sum_{j=1}^a w^\sigma_j$$
is the minimal vector in the orbifold coset labeled by $\sigma$.
\subsubsection{Simplification}
 To find $\ell(x^\sigma)$ and $\sk(x^\sigma)$, we simplify our expression for $x_i^\sigma$.  The following definition helps.
\begin{definition}
For $i<j$, define $\des_{ij}$ to be the number of descents between $i$ and $j$.  That is:
$$\des_{ij}(\sigma)=\left|\left\langle k\in\DES(\sigma)\big| i\leq k < j\right\}\right|=\sum_{k=i}^{j-1}\des_k(\sigma)$$
\end{definition}

With this definition,
$$w_j=\frac{\sigma_j}{a}+\des_{1,j}(\sigma)$$
and so
\begin{align*}
\sum_{j=1}^a w_j &=\frac{1}{a}\sum_{i=1}^a \sigma_i +\sum_{i=1}^a \des_{1,i} \\
 &=\frac{a+1}{2}+\sum_{i=1}^{a-2} (a-i)\des_i(\sigma)
\end{align*}

Thus,
$$x^\sigma_j=\frac{\sigma_j}{a}+\des_{1j}(\sigma)-\frac{a+1}{2a}-\frac{1}{a}\sum_{i=1}^{a-2} (a-i)\des_i(\sigma)$$

\subsubsection{Length of $x^\sigma$}
We compute (recalling the convention $\sigma_a=a$):

\begin{align*}
\ell(x^\sigma) &= a x^\sigma_a-\frac{a-1}{2} \\
&=a+a\sum_{i=1}^{a-2} \des_i(\sigma)-\frac{a+1}{2}-\sum_{i=1}^{a-2} (a-i)\des_i(\sigma)-\frac{a-1}{2} \\
&=\sum_{i=1}^{a-2} i\des_i(\sigma) \\
&=\maj(\sigma)
\end{align*}

\subsubsection{Skew length of $x^\sigma$}
We have
\begin{align*}
\sk (x^\sigma) &=\sum_{1\leq i <j \leq a} \left\langle v^\sigma_j-v^\sigma_i\right\rangle \\
&=\sum_{1\leq i <j \leq a} \left\langle\frac{\sigma_j-\sigma_i}{a}+\des_{ij}(\sigma)\right\rangle \\
&=\sum_{1\leq i<j\leq a} \des_{ij}(\sigma)-\delta(\sigma_j<\sigma_i)
\end{align*}
Observe
$$\sum_{1\leq i<j\leq a} \delta(\sigma_j<\sigma_i)=\inv(\sigma).$$  
and
$$\sum_{1\leq i<j\leq a} \des_{ij}(\sigma)=\sum_{k=1}^{a-2} k(a-k)\des_k(\sigma)$$
since for $\des_k$ to appear in $\des_{ij}$ we need $1\leq i\leq k$ and $j<k\leq a$, and so $\des_k$ appears in $k(a-k)$ different $\des_{ij}$.

Thus, we have shown
$$\sk(x^\sigma)=\sum_{k=1}^{a-2} k(a-k)\des_k(\sigma)-\inv(\sigma)=\siz(\sigma)$$

\qed

\subsection{Proof of Lemma \ref{lem:sizmaj2}}

Before we prove Lemma \ref{lem:sizmaj2}, we introduce a family of codes for permuations that we call \emph{factorization codes}; our proof uses a specific factorization code we call the left-decreasing factorization code.

\begin{definition}
A \emph{valid sequence of length $n$} is a sequence of integers $a_i, 1\leq i\leq n $ such that $0\leq a_i<i$.  Let $\VS_n$ denote the set of valid sequences; clearly $|\VS|=n!$.

A \emph{permutation code} is a bijection $\phi:\VS_n\to S_n$.
\end{definition}

In section \ref{sec:factorization} we introduce a family of permutation codes we call \emph{factorization codes}; in particular, this family includes the \emph{left-decreasing factorization code} $\LD$.  

Lemma \ref{lem:sizmaj2} then reduces to showing:
\begin{lemma} \label{lem:LDweights}
For a valid sequence $a\in \VS_n$, we have:
\begin{align*}
\maj(LD(a))&=\sum a_i \\
\siz(LD(a))&=\sum (n+1-i) a_i
\end{align*}
\end{lemma}

\subsubsection{Factorization codes} \label{sec:factorization}

Factorization codes rest on the following simple observation.  Let $C_k\in S_k$ be any $k$-cycle.  Then $\{C_k^i\}, 0\leq i < k$ form a family of representatives for the (left or right) cosets of $S_{k-1}\subset S_k$.   

\begin{definition}
A \emph{family $C$ of $k$-cycles} is a sequence $C_k, k\in \N$, with $C_k\in S_K$ a $k$-cycle.

The \emph{right factorization} code associated to a family of $k$-cycles $C_k$ is the sequence of maps $R^C_n:\VS_n\to S_n$ defined by
$$R_n(a)=\alpha_k=C_2^{a_2}C_3^{a_3}\cdots C_n^{a_n}$$

Similarly, the \emph{left factorization} code associated to a family of $k$-cycles $C_k$ is the the sequence of maps $L^C_n:\VS_n\to S_n$ defined by
$$L^C_n(a)=C_n^{a_n}C_{n-1}^{a_{n-1}}\cdots C_2^{a_2}$$
\end{definition}

That the left and right factorization codes are in fact permutation codes follows easily from the observation using induction on $n$.

There are two ``obvious'' families of $k$-cycles:  \emph{increasing} cycles $C^+_k=(1,2,3,\dots,k)$, and the \emph{decreasing} cycles $C^-_k=(k,k-1,k-2,\dots,1)$.  

Thus, the left-decreasing factorization code $L^-_n$ is the bijection that sends $0\leq a_i<i$ to to 
$$L^-_n(a)=(C_n^-)^{a_n}(C_{n-1}^-)^{a_{n-1}}\dots (C_2^-)^{a_2}$$

\subsubsection{Multiplication by $C_k^-$} \label{sec:factorizationinductuction}

We now inductively prove Lemma \ref{lem:LDweights} giving $\maj$ and $\siz$ of a permutation in terms of its left decreasing factorization code.

Clearly Lemma \ref{lem:LDweights} holds on the identity permutation, where all $a_i=0$.   Thus we must show that in such a factorization, multiplying by $C^-_k$ raises $\maj$ by one and $\siz$ by $(n+1-k)$.

To do this, we must determine what multiplication by $C_k^-$ does to the set $\DES$ of descents.  When multiplying by $C^-_k$, we have not yet permuted the elements $(k+1), (k+2),\dots, n$, and so $\DES\subset \{1,\dots, k-1\}$. As $C_k$ decreases $2,\dots, j$ by 1, any comparisons involving two of these elements remains unchanged; hence, the only descents multiplying by $C^-_k$ could change are those involving 1, which it changes to $k$.

Suppose that in the one-line notation of $\sigma$ the 1 is in position $j$; then $j-1$ is a descent (unless $j=1$), and $j$ is not a descent.  After we multiply by $c_k$, the 1 changes to a $k$, and so now $j-1$ is not a descent,but $j$ is.

Thus, multiplying by $C_j$ either increases a descent by one, or creates a new descent at 1.  In either case, the major index increases by one.

We now investigate the effect of multipication by $C_k$ on $\siz$, supposing that 1 is in position $j$.  We first determine the change in the first term in $\siz$ (the sum over descents), and then determine the change this makes to the second term $\inv$.

A descent at $j-1$ contributes $$(n+1-(j-1))(j-1)=nj-j^2+3j-2$$ to $\siz$; a descent at $j$ contributes $$(n+1-j)j=nj-j^2+j$$
and thus multiplying by $C_k^-$ when $1$ is in position $j<k$ increases the first term of $\siz$ by $2-2j$.

We now turn to the inversions.  It is clear that the only inversions that change are those that were comparing $1$.  Before we multiply by $C_k^-$, $1$ is in position $j$, and the $j-1$ pairs $(i,j), 1\leq i \leq j-1$ are inversions, and none of the $k-j$ pairs $(j,\ell), j+1\leq \ell \leq k$ are inversions.  Multiplying by $C_k^-$, changes position $j$ to $k$.  Now none of the pairs $(i,j)$ are inversions, and all of the pairs $(j,\ell)$ are inversions.  Thus, $\inv$ increases by $k-2j+1$.  

Multiplying by $C_k^-$ when $1$ is in position $j<k$ will change $\siz$ by $$n-2j+2-(k-2j+1)=n-k+1$$ as desired.

\bibliographystyle{plain}
\bibliography{Armstrong}

\end{document}